\documentclass[11pt, notitlepage]{article}
\usepackage{amssymb,amsmath,comment}
\usepackage{amsthm}
\allowdisplaybreaks
\catcode`\@=11 \@addtoreset{equation}{section}
\def\thesection{\arabic{section}}

\def\theequation{\thesection.\arabic{equation}}
\catcode`\@=12
\usepackage{colortbl}
\usepackage{hyperref}
\usepackage[mathscr]{eucal}
\usepackage{epsf}
\usepackage{esint}
\usepackage{a4wide}
\def\R{\mathbb{R}}

\DeclareMathOperator*{\esssup}{ess\,sup}
\DeclareMathOperator*{\essinf}{ess\,inf}

\newcommand{\De} {\Delta}
\newcommand{\la} {\lambda}

\newcommand{\noi} {\noindent}

\usepackage[all]{xy}
\catcode`\@=11
\def\theequation{\@arabic{\c@section}.\@arabic{\c@equation}}
\catcode`\@=12

\newtheorem{Theorem}{Theorem}[section]
\newtheorem{Lemma}[Theorem]{Lemma}
\newtheorem{prop}[Theorem]{Proposition}
\newtheorem{Corollary}[Theorem]{Corollary}
\newtheorem{Remark}[Theorem]{Remark}
\newtheorem{Definition}[Theorem]{Definition}
\newtheorem{Example}{Example}

\begin{document}
{\vspace{0.01in}}

\title{Symmetry and Approximate Symmetry for Solutions of Mixed Local-Nonlocal Singular Equations}
\author{Sanjit Biswas\\
Department of Mathematical Sciences\\
Indian Institute of Science Education and Research Berhampur}

\maketitle

\begin{abstract}\noindent
In this article, we establish radial symmetry for positive weak solutions of a class of mixed local-nonlocal equations with possibly singular nonlinearity via the moving plane method. Furthermore, we provide a quantitative version of Gidas-Ni-Nirenberg type theorem for mixed local-nonlocal equations. In this regard, we establish a weak Harnack-type inequality and an analogue of the Alexandroff-Bakelman-Pucci inequality in the mixed nonhomogeneous setting with a lower order term, which appear to be new. To the best of our knowledge, this paper initiates the study of the quantitative properties of solutions to mixed problems.
\end{abstract}

\maketitle

\noi {Keywords: Mixed local-nonlocal problems, singular nonlinearity, symmetry, approximate symmetry, moving plane method.}

\noi{\textit{2020 Mathematics Subject Classification: 35M10, 35M12, 35J75, 35R06, 35R11}

\bigskip

\tableofcontents

\section{Introduction and main results}
This article is devoted to the study of symmetry and approximate symmetry of weak solutions to a class of semilinear equations of the form:
\begin{align}\label{ME1}
        -\De u+(-\De )^su=f(x,u) \text{ in }\Omega\setminus\Gamma,
    u>0 \text{ in }\Omega\setminus\Gamma\text{ and }u=0\text{ in }\R^n\setminus\Omega,
\end{align}
where $\Omega\subset\R^n$ is a bounded domain with $C^{1,1}$-type boundary and $u$ is singular on the closed set $\Gamma\subset\Omega$. Here, $\De$ denotes the usual Laplacian operator and the operator $(-\De)^s$ is known as the fractional Laplacian and defined by
 \begin{equation*}
     (-\De)^su(x):=P.V.\int_{\R^n} \frac{(u(x)-u(y))}{|x-y|^{n+2s}}\;dy,
 \end{equation*}
 where $s\in(0,1)$ and P.V. stands for the principal value integral. Note that some definitions of the fractional Laplacian include a normalization constant, but for simplicity, we omit it here. The equations involving mixed local and nonlocal operators have numerous applications in stochastic processes, image processing, biology, etc.; for more details, we refer to \cite{serena} and the references therein. The study of mixed operators has been a central focus of many researchers; for instance, see \cite{ Valdinoci23, BiagCCM, BG2, BG26, Mingioni24, ValdinociProc, GJ, ValdinociJDE25} and the references therein. In this article, we plan to establish several quantitative and qualitative properties of weak solutions to equation (\ref{ME1}). To this concern, we assume the following conditions on the continuous function $f:\overline{\Omega}\setminus\Gamma\times (0,\infty)\to \R$:
\begin{itemize}
    \item[($\mathscr{F}_1$)] For $0<t\leq s\leq b<\infty$ and for every $\omega\Subset \overline{\Omega}\setminus\Gamma$, there exists $\mathscr{K}_1=\mathscr{K}_1(b,\omega)>0$ such that $$f(x,s)-f(x,t)\leq \mathscr{K}_1(s-t)\text{ for all }x\in \omega.$$
    \item[($\mathscr{F}_2$)] For every $I\Subset (0,\infty)$, and $\omega\Subset \Omega\setminus\Gamma$, there exists $\mathscr{K}_2=\mathscr{K}_2(I,\omega)>0$ such that $$|f(x,s)-f(x,t)|\leq \mathscr{K}_2|s-t|\text{ for all }s,t\in I,\;x\in \omega.$$
\end{itemize}
The following singular equation is one of the prototypes of equation (\ref{ME1}) with $\Gamma=\phi$:
\begin{align}\label{A1}
    -\De u+\epsilon(-\De )^su=\gamma u^{-\delta}+h(u) \text{ in } \Omega,
    u>0 \text{ in }\Omega \text{ and }u=0\text{ in }\R^n\setminus \Omega,
\end{align}
where $\delta>0,\epsilon>0,\gamma>0$ and $h:[0,\infty)\to\R$ is a locally Lipschitz continuous. Here, the positivity of $\delta$ leads to a blow-up of the nonlinearity near the origin, a phenomenon referred to as singularity. When $h$ satisfies the subcritical or critical growth conditions, the existence of multiple solutions to equation (\ref{A1}) was addressed in \cite{ BV24, Sanjit, PG}, while purely local and purely nonlocal cases were studied in \cite{ADICCM, Haitao} and \cite{ST16}.

The symmetry and monotonicity of solutions to purely local equations have been widely studied in the literature. In the seminal work of Gidas, Ni, and Nirenberg \cite{Gidas}, it was shown that every classical solution $u\in C^2(\Omega)\cap C(\overline{\Omega})$ to the purely local equation 
\begin{align}\label{A2}
    -\Delta u=h(x,u)\text{ in }\Omega, u>0 \text{ in }\Omega \text{ and }u=0\text{ on }\partial\Omega
\end{align}
is radially symmetric and monotonically decreasing in the radial direction, provided that $\Omega$ is a ball in $\R^n$ and $h:\overline{\Omega}\times[0,\infty)\to\R$ is of type $C^1$. For the whole space $\R^n$, the symmetry and monotonicity of solutions to a class of semilinear equations were obtained in \cite{CaffCPAM, Chen91}. The Gidas-Ni-Nirenberg type results for quasilinear equations can be found in \cite{DP98, DS}, while purely local and mixed local-nonlocal cases were addressed by Chen et al. \cite{ChenLi17} and Valdinoci et al. \cite{BV21}, respectively. 

Very recently, a quantitative version of the Gidas-Ni-Nirenberg theorem was obtained by Cozzi et al. in \cite{Cozzi24}. More precisely, the authors quantitatively measure the deviation of classical solutions to the equation
\begin{align}\label{MM2}
         -\Delta u=k(x)g(u)\text{ in }\Omega,
        u>0\text{ in }\Omega \text{ and }u=0 \text{ on }\partial\Omega,
\end{align}
from radial function, where $\Omega$ is the unit ball in $\R^n$, $g:[0,\infty)\to[0,\infty)$ is a locally Lipschitz continuous function, and the function $k:\overline{\Omega}\to (0,\infty)$ belongs to $C^1(\overline{\Omega})$. Furthermore, in \cite{Cozzi25}, the authors established several quantitative estimates for classical solutions to equation (\ref{MM2}) in the whole space $\R^n$ using the moving plane technique, while the quasilinear analogues of these results were addressed in \cite{Gatti}. More qualitative results can be found in \cite{Li25, Serena25}.

The radial symmetry and monotonicity properties of solutions to equations involving singular nonlinearities of the form 
\begin{align}\label{ME00}
         -\Delta u=u^{-\delta}+h(u)\text{ in }\Omega,
        u>0\text{ in }\Omega\text{ and }u=0 \text{ on }\partial \Omega,
\end{align}
have also been extensively studied, where $\delta>0$ and $h:[0,\infty)\to\mathbb{R}$ is a locally Lipschitz continuous function. The presence of the singular source term $u^{-\delta}$ causes the nonlinearity to lose local Lipschitz continuity near zero, and as a result, classical symmetry results are no longer directly applicable to solutions of \eqref{ME00}. For the existence and uniqueness, we refer the reader to \cite{Boccardo10, Sciunzi16, CRT, LM} and the references therein, while the purely nonlocal problem was addressed in \cite{Sciunzi17}. In \cite{SciunziJDE13, SciunziNA}, Sciunzi et al. established the Gidas–Ni–Nirenberg type results for solutions of (\ref{ME00}) using the moving plane technique. A similar result for the purely nonlocal equation was investigated by Giacomoni et al. in \cite{Rakesh20}.
     
In the mixed local–nonlocal setting, Garain and Anthal in \cite{GG25} studied the equation
  \begin{align}
       -\Delta u+(-\Delta)^su=u^{-\delta}+h(u)\text{ in }\Omega,
        u>0\text{ in }\Omega\text{ and }u=0 \text{ on }\mathbb{R}^n\setminus \Omega,
  \end{align}
  and showed the radial symmetry of weak solutions to the equation under the assumptions that $h:[0,\infty)\to [0,\infty)$ is locally Lipschitz, non-decreasing and $h(u)>0$ when $u>0$. Their approach relies on a decomposition technique combined with the moving plane method. In this article, we provide an alternative proof of their result that accommodates a broader class of singular nonlinearities. Our first theorem can be stated as follows:
  \begin{Theorem}[Symmetry]\label{T1}
    Let $\Omega\subset\R^n$ be a smooth domain that is convex with respect to the $x_1$-direction and symmetric with respect to the hyperplane $\{x_1=0\}$. Assume that $\Gamma\subset \{x\in\Omega: x_1=0\}$ is a closed set such that $\Gamma$ is a point when $n=2$ and $\mathrm{cap}_2(\Gamma)=0$ when $n\geq 3$. Furthermore, we assume that $f:\overline{\Omega}\setminus\Gamma\times(0,\infty)\to\R$ is a continuous function which satisfies the hypotheses $(\mathscr{F}_1), (\mathscr{F}_2)$, and 
    \begin{align}\label{mp}
        f(x_1,x_2,...,x_n,u)\leq f(y_1,x_2,...,x_n,u)\text{ for all }0<y_1<x_1,
    \end{align}
    and 
    \begin{align}\label{sp}
        f(x_1,x_2,...,x_n,u)=f(-x_1,x_2,...,x_n,u)\text{ for every }x=(x_1,x_2,...,x_n)\in\Omega.
    \end{align}
    Suppose that $u\in H^1_{\mathrm{loc}}(\Omega\setminus\Gamma)\cap C(\overline{\Omega}\setminus\Gamma)$ is a weak solution to \eqref{ME1} in the sense of Definition \ref{def1}. Then $u$ is symmetric with respect to the hyperplane $\{x_1=0\}$. 
\end{Theorem}
  \begin{Remark}
      The function $f$ is allowed to take negative values, and the map $s\to f(x,s)$ is not assumed to be monotonically decreasing---assumptions that played a central role in \cite[Theorem 2.14]{GG25}. Our result therefore extends \cite[Theorem 2.14]{GG25}.
  \end{Remark}
 \begin{Remark}
      Let us briefly discuss our assumptions on the singular set $\Gamma$. For purely local case, in \cite[Theorem 1.3]{SciunziJDE18}, the authors have assumed that the singular set $\Gamma$ has zero 2-capacity when $n\geq 3$ and is a point when $n=2$, while in \cite[Theorem 1.1]{SciunziJMPA17} the author assumed that $\Gamma\subset\{(x_1,x_2,...,x_n)\in \Omega:x_1=0\}$ is contained in a submanifold of dimension $d\leq n-2$ when $n\geq 3$ and is a point when $n=2$. Subsequently, in \cite[Example 1, Example 2, page 630]{BV20}, it was shown that the zero 2-capacity assumption is somehow sharp for purely local case. 
      
      In the mixed case, since local operator has a higher order, motivated by \cite[Theorem 1.3]{SciunziJDE18}, we assume the singular set $\Gamma$ has zero $2$-capacity for $n\geq 3$ and it is a point when $n=2$. 
      
    This condition on $\Gamma$ implies that there exists a sequence of test functions supported in a neighborhood of $\Gamma$ with the energy converges to zero, which is crucial for initiating the moving plane technique. Furthermore, the Example \ref{necessary} demonstrates that our assumption is somehow sharp. Since there exists a set with zero $s$-capacity but positive $2$-capacity, the condition $\mathrm{cap}_2(\Gamma)=0$ can not be replaced by $\mathrm{cap}_s(\Gamma)=0$ for our argument. 
  \end{Remark}
To establish this result, we mainly use the moving plane technique (see \cite{BV21}). However, it is worth emphasizing that the solutions belong to the local Sobolev space, and hence the moving plane technique can not be applied directly. To overcome this, we adopt the methodologies introduced in \cite{SciunziJDE18, Luigi18}. Specifically, to start the moving plane procedure, we show that a suitably constructed auxiliary function belongs to the Sobolev space (see Lemma \ref{MLemma}), which requires several careful estimates. Finally, using the moving plane technique, we obtain our desired result.

This theorem leads us to the following important corollary.
\begin{Corollary}\label{Cor1}
Assume that $\Omega\subset\R^n$ is a ball and $f:\overline{\Omega}\times (0,\infty)\to \R$ is a continuous function that satisfies hypotheses $(\mathscr{F}_1),(\mathscr{F}_2)$. Assume further that $f(x,u)=f(|x|,u)$ for all $x\in\Omega$ and $f$ is radially non-increasing. Suppose that $u\in H^1_{\mathrm{loc}}(\Omega)\cap C(\overline{\Omega})$ is a weak solution to the equation
    \begin{align}\label{ME3}
    \begin{cases}
        -\De u+(-\De )^su=f(x,u)\text{ in }\Omega,\\
    u>0 \text{ in }\Omega \text{ and }u=0\text{ in }\R^n\setminus \Omega,
    \end{cases}
\end{align}
 in the sense of Definition \ref{def1}. Then $u$ is radially symmetric in $\Omega$.
\end{Corollary}
The following corollary ensures the radial symmetry of solutions to variable exponent singular problems, which is itself a new result. 
\begin{Corollary}\label{Cor2}
    Let $\Omega\subset\R^n$ be a ball and $g:[0,\infty)\to \R$ be a locally Lipschitz continuous function. Suppose $u\in H^1_{\mathrm{loc}}(\Omega)\cap C(\overline{\Omega})$ is such that $0<u(x)\leq 1$ and weakly satisfies the following equation:
\begin{align}\label{ME2}
    \begin{cases}
        -\De u+(-\De )^su=\frac{1}{u^{\delta(x)}}+g(u)\text{ in }\Omega,\\
    u>0 \text{ in }\Omega \text{ and }u=0\text{ in }\R^n\setminus \Omega,
    \end{cases}
\end{align}
 where $\delta:\overline{\Omega}\to (0,\infty)$ is radial and non-increasing in the radial direction. Then $u$ is radially symmetric in $\Omega$. 
\end{Corollary}
 Our second main theorem provides a quantitative version of Corollary \ref{Cor1} and it can be stated as follows.
\begin{Theorem}[Approximate Symmetry]\label{T2}
     Assume that $\Omega\subset\R^n$ is the unit ball. Suppose $f:\overline{\Omega}\times[0,\infty)\to\R$ is a continuous function of the form $f(x,u)=k(x)g(u)$, where $k\in C^1(\overline{\Omega};(0,\infty))$ and $g\in \mathrm{Lip}_{\mathrm{loc}}([0,\infty);[0,\infty))$. Assume that $u\in H^1_0(\Omega)\cap C^1(\overline{\Omega})$ is a weak solution to \eqref{ME3} such that 
    \begin{align}\label{lowerbound}
        \frac{1}{C_0}(1-|x|)\leq u(x)\leq C_0 \text{ in } \Omega,
    \end{align}
    for some $C_0\geq 1.$ Then there exist positive constants $C>0$, $\gamma\in(0,1)$, depending only on $n,\; s,\; C_0,\; \|g\|_{C^{0,1}[0,C_0]},$ $\|k\|_{L^\infty(\overline{\Omega})}$ such that  $$|u(x)-u(y)|\leq C \mathrm{def}(k)^\gamma,\text{ for all }x,y\in \Omega \text{ such that } |x|=|y|.$$
Here 
    \begin{align}\label{def}
        \mathrm{def}(k)=\|\nabla^Tk\|_{L^\infty(\Omega)}+\|\partial^+_rk\|_{L^\infty(\Omega)}
    \end{align}
    with $\nabla^T=\nabla-\frac{x}{|x|}\partial_r$ and $\partial_r=\frac{x}{|x|}\cdot\nabla$.
\end{Theorem}
\begin{Remark}
    If the function $k$ is radial and radially decreasing, then $\mathrm{def}(k)=0$. Consequently, this estimate is consistent with the conclusion of Corollary \ref{Cor1}. 
\end{Remark}
In order to prove this theorem, we adopt the ideas developed in \cite{Cozzi25, Cozzi24, Gatti}. However, the absence of a weak Harnack inequality and an analogue of Alexandroff-Bakelman-Pucci-type estimate for the mixed local-nonlocal nonhomogeneous equations with lower order terms prevent us to apply their method directly. Inspired by \cite{BV24, Garain25, GJ}, we establish a weak Harnack-type inequality (Proposition \ref{WHI}) for this setting, which is based on the Giorgi-Nash-Moser theory. An iteration technique is used to obtain an ABP-type inequality, which is reminiscent of \cite{Gatti}. Subsequently, using the moving plane technique, we conclude our result.

This article is organized in the following way. In Section 2, we present all the functional spaces and several useful results. Section 3 is devoted to the proof of several preliminary results for Theorem \ref{T1}. Finally, in Sections 4 and 5, we provide a detailed proof of our main results. \\
\textbf{Notations:} For the rest of the paper, unless otherwise mentioned, we will use the following notations and assumptions:
\begin{itemize}

    \item $\Omega\subset\mathbb{R}^n$ with $n\geq2$ is a bounded domain with smooth boundary.

     
    \item For open sets $\omega$ and $\Omega$ of $\mathbb{R}^n$, by notation $\omega\Subset\Omega$, we mean that $\overline{\omega}$ is a compact subset of $\Omega$.
    
    \item For a measurable set $A\subset\mathbb{R}^n$, $|A|$ denotes the Lebesgue measure of $A$. Moreover, for a function $u:A\to\R$, we define $u^+:=\max\{ u, 0\}$ and $u^-:=\max\{-u, 0\}.$ 

    \item An $\epsilon$-neighborhood of a set $A\subset\R^n$ is denoted by $B_\epsilon(A)$ and defined as $$B_\epsilon(A):=\{x\in\R^n: d(x,A)<\epsilon\}.$$ 

    \item We say $J(\epsilon)\geq o(\epsilon)$ if $\lim_{\epsilon\to 0} J(\epsilon)\geq 0$.
    
    \item For a measurable set $B\subset\R^n$, $\fint_B$ denotes the average $\frac{1}{|B|}\int_B.$
    
    \item $C$ denotes a positive constant, whose values may change from line to line or even in the same line.

    \item For an integer $m\geq 1$, $C^m(\Omega;\R)$ is the set of all $m^{th}$-order continuously differentiable functions from $\Omega$ to $\R$.

    \item $\mathrm{Lip}_{\mathrm{loc}}(\Omega;\R)$ denotes the set of all real-valued locally Lipschitz continuous functions on $\Omega$.
\end{itemize}
\section{Preliminaries}
\subsection{Functional analytical setting}
The Sobolev space $H^{1}(\Omega)$ is defined by
$$H^{1}(\Omega):=\{ u\in L^2(\Omega): \nabla u\in L^2(\Omega,\R^n) \}$$
equipped with the norm
$$
\|u\|_{H^{1}(\Omega)}=\left(\int_{\Omega}|u(x)|^2\,dx+\int_{\Omega}|\nabla u|^2\;dx\right)^\frac{1}{2},
$$
where $\nabla u=\Big(\frac{\partial u}{\partial x_1},\ldots,\frac{\partial u}{\partial x_n}\Big)$. We say $u\in H^1_{\mathrm{loc}}(\Omega)$ if $u\in H^1(\omega)$ for every $\omega\Subset\Omega$. The fractional Sobolev space $W^{s,2}(\Omega)$ for $0<s<1$, is defined by
$$
W^{s,2}(\Omega)=\Bigg\{{u:\Omega\to\mathbb{R}:\,}u\in L^2(\Omega),\,\frac{|u(x)-u(y)|}{|x-y|^{\frac{n}{2}+s}}\in L^2(\Omega\times \Omega)\Bigg\}
$$
under the norm
$$
\|u\|_{W^{s,2}(\Omega)}=\left(\int_{\Omega}|u(x)|^2\,dx+\int_{\Omega}\int_{\Omega}\frac{|u(x)-u(y)|^2}{|x-y|^{n+2s}}\,dx\,dy\right)^\frac{1}{2}.
$$
We refer to \cite{Hitchhikersguide} and the references therein for more details on fractional Sobolev spaces. Due to the mixed behavior of our equations, we consider the space
$$
H_0^{1}(\Omega)=\{u\in W^{1,2}(\mathbb{R}^n):u=0\text{ in }\mathbb{R}^n\setminus\Omega\}
$$
under the norm
$$
\|u\|_{H_0^{1}(\Omega)}=\left(\int_{\Omega}|\nabla u|^2\,dx+\int_{\mathbb{R}^{n}}\int_{\mathbb{R}^{n}}\frac{|u(x)-u(y)|^2}{|x-y|^{n+2s}}\, dx dy\right)^\frac{1}{2}.
$$
{The following Lemma guaranties that the norm $\|u\|_{H_0^{1}(\Omega)}$ defined above is equivalent to the norm $\|u\|=\|\nabla u\|_{L^2(\Omega)}$.}
\begin{Lemma}(see \cite[Lemma $2.1$]{Silva})\label{Gagliardo}
For $0<s<1$, there exists a constant $C=C(n,s,\Omega)>0$ such that
\begin{equation}\label{locnonsem}
\int_{\mathbb{R}^n}\int_{\mathbb{R}^n}\frac{|u(x)-u(y)|^2}{|x-y|^{n+2s}}\,dx\,dy\leq C\int_{\Omega}|\nabla u|^2\,dx
\end{equation}
for every $u\in H_0^{1}(\Omega)$.
\end{Lemma}
For subsequent Sobolev embedding, refer, for instance, to \cite{Evans}.
\begin{Lemma}\label{emb}
The embedding operators
\[
H_0^{1}(\Omega)\hookrightarrow
\begin{cases}
L^t(\Omega),&\text{ for }t\in[1,2^{*}],\text{ if }n>2,\\
L^t(\Omega),&\text{ for }t\in[1,\infty),\text{ if }n=2
\end{cases}
\]
are continuous. Moreover, they are compact except for $t=2^*$ if $n>2$. Here $2^*=\frac{2n}{n-2}$ if $n>2$.
\end{Lemma}
\begin{Definition}[Capacity measure]
    For a compact set $K\subset\Omega$, the $2$-capacity of $K$ is denoted by $\mathrm{cap_{\text{$2$}}(K)}$ and defined as
$$\mathrm{cap\text{$_2$}(K)}:=\inf \left \{ \int_\Omega|\nabla\phi|^2 dx : \phi\in C^\infty_c(\Omega), \phi\geq \chi_K \right \},$$
where $$\chi_K(x):=\begin{cases}
    1\text{ if } x\in K,\\
    0 \text{ otherwise.}
\end{cases}$$
\end{Definition}
Finally, $2$-capacity of any subset $B$ of $\Omega$ is defined by the standard way. For more details, we refer the reader to \cite{EvansG}. Now, we define the notion of solutions.
\begin{Definition}[Weak solution]\label{def1}
    A function $u\in  H^1_{\mathrm{loc}}(\Omega\setminus\Gamma)\cap L^1(\R^n)$ is said to be a weak solution to equation (\ref{ME1}) if $u>0 \text{ in }\Omega\setminus\Gamma,\;u=0\text{ in }\R^n\setminus\Omega$ and
    $$\int_\Omega \nabla u\cdot\nabla\phi\;dx+\int_{\R^n}\int_{\R^n}\frac{(u(x)-u(y))(\phi(x)-\phi(y))}{|x-y|^{n+2s}}\;dy\;dx=\int_\Omega f(x,u)\phi\;dx,$$
    for every $\phi\in C^1_c(\Omega\setminus\Gamma)$.
\end{Definition}

 Before we proceed further, motivated by \cite[Example 2]{BV20}, we construct an example which shows that the condition $\mathrm{cap}_2(\Gamma)=0$ in Theorem \ref{T1} is necessary.
\begin{Example}\label{necessary}
    Let $\Omega=B_1(0)\subset \R^2,\;\Gamma=\{0\}\times[-\frac{1}{2},\frac{1}{2}]$ and $s\in(0,\frac{1}{2})$. For $m\geq 2$, we define $$R_m:=\{(x_1,x_2)\in \Omega: (mx)^2+(2y)^2\leq 1\},$$
    and $\Omega_m:=\Omega\setminus R_m$. Note that $\Omega_m$ has smooth boundary. Let $\varphi_m\in C^\infty_c(\Omega)$ be such that $0\leq \varphi_m\leq 2$, $\varphi_m(\frac{1}{m},0)=1$ and $\varphi_m(-\frac{1}{m},0)=2$. Due to \cite[Proposition 2.5]{Silvestre}, we have $\Delta\varphi_m-(-\Delta)^s\varphi_m\in C^{0,1-2s}(\R^2)$. Now, we consider the following equation:
    \begin{align}\label{NC}
        \begin{cases}
            -\Delta z+(-\Delta)^sz=0\mbox{ in }\Omega_m,\\
            z=\varphi_m\mbox{ in }\Omega_m^c.
        \end{cases}
    \end{align}
    Since $\Omega_m$ is smooth, $0<s<\frac{1}{2},\mbox{ and } \Delta\varphi_m-(-\Delta)^s\varphi_m\in C^{0,1-2s}(\Omega_m)$, using \cite[Theorem 2.8]{Valdinoci23}, there exists an unique solution $v_m\in C(\R^2)\cap L^\infty(\R^2)\cap C^{2,1-2s}(\overline{\Omega_m})$ of the equation 
    \begin{align}
        \begin{cases}
            -\Delta v_m+(-\Delta)^sv_m=\Delta\varphi_m-(-\Delta)^s\varphi_m\mbox{ in }\Omega_m,\\
            z=0\mbox{ in }\Omega_m^c.
        \end{cases}
    \end{align}
    Thus, $u_m:=v_m+\varphi_m\in C^{2,1-2s}(\overline{\Omega_m})$ is the unique solution of (\ref{NC}). Using the strong maximum principle \cite[Theorem 1.3]{ValdinociCPDE}, we can conclude that $0\leq u_m\leq 2$ in $\overline{\Omega_m}$. Now, for $l\geq 2$, we define a set $O_l$ as follows:
    $$O_l:=\{x\in \Omega: d(x,\partial(\Omega\setminus\Gamma))>\frac{1}{l}\}.$$
We observe that for every fixed $l\geq 2$, there exists $m_l\in\mathbb{N}$ such that $\overline{O_l}\subset \Omega_m$ for every $m\geq m_l$. Hence, the sequence $\{u_m\}_{m\geq m_l}$ is uniformly bounded in $\overline{O_l}$. Furthermore, applying \cite[Theorem 1.6]{ValdinociJDE25} we obtain that $\|u_m\|_{C^{2,1-2s}(\overline{O_l})}\leq C$, for some constant $C>0$, independent of $m$.
Using Arzela-Ascoli theorem together with Cantor's diagonalization argument we obtain a subsequence $\{u_{m_j}\}_{j\geq 1}$ of $\{u_m\}_{m\geq m_l}$ such that $u_{m_j}\to u$ and $\nabla u_{m_j}\to \nabla u$ uniformly on compact subsets of $\Omega\setminus\Gamma$. Thanks to the Lebesgue dominated convergence theorem, we have 
\begin{align}
        \begin{cases}
            -\Delta u+(-\Delta)^s u=0 \mbox{ in }\Omega\setminus\Gamma,\\
            u=0\mbox{ in }\Omega^c \mbox{ and } u\geq 0\mbox{ in }\Omega\setminus\Gamma.
        \end{cases}
    \end{align}
Thanks to \cite[Theorem 1.3]{ValdinociCPDE} and \cite[Theorem 1.6]{ValdinociJDE25}, one has $u>0$ in $\Omega\setminus\Gamma$ and $u\in C(\overline{\Omega}\setminus\Gamma)$. We claim that $u$ is not symmetric with respect to the hyperplane $x_1=0$. To this concern, we again use \cite[Theorem 1.6]{ValdinociJDE25} and the fact $\varphi_{m_j}(-\frac{1}{m_j},0)=2$ to obtain 
\begin{align}
    |u_{m_j}(x)-2|=|u_{m_j}(x)-u_{m_j}(-\frac{1}{m_j},0)|\leq C|x_1
+\frac{1}{m_j}|,
\end{align}
$\mbox{ for every }x=(x_1,0)\in\Omega \mbox{ with }x_1<-\frac{1}{m_j}.$ Taking $j\to\infty$, we deduce
$|u(x)-2|\leq C|x_1|$ and hence $$\lim_{x_1\to 0^-}u(x_1,0)=2.$$
Since $\varphi_{m_j}(\frac{1}{m_j},0)=1$, by similar way we obtain $$\lim_{x_1\to 0^+}u(x_1,0)=1.$$
Thus, $u$ is not symmetric with respect to the hyperplane $x_1=0$. It is noteworthy that $\mathrm{cap}_2(\Gamma)>0$.
\end{Example}

We now define several notation that will be used to obtain the weak Harnack inequality. For $0<s<1$, the tail space $\mathcal{L}^s(\R^n)$ is defined by 
$$\mathcal{L}^s(\R^n):=\{u\in L^1_{\mathrm{loc}}(\R^n): \int_{\R^n} \frac{|u(y)|}{(1+|y|)^{n+2s}}\;dy<\infty\}.$$
Moreover, the Tail of a function $u$ with respect to the ball $B_r(x_1)$ is denoted by $\mathrm{Tail}(u;x_1,r)$ and defined as 
\begin{align}\label{Tail}
    \mathrm{Tail}(u;x_1,r):=r^2\int_{\R^n\setminus {B_r(x_1)}}\frac{|u(y)|}{|x_1-y|^{n+2s}}\;dy.
\end{align}

\begin{Definition}
    A function $u\in H^1_{\mathrm{loc}}(\Omega)\cap\mathcal{L}^s(\R^n)$ is said to be a supersolution (subsolution) to $$-\Delta u+(-\Delta)^su+c(x)u=g\text{ in }\Omega$$ if the following inequality holds:
$$\int_\Omega \nabla u\cdot\nabla\phi\;dx+\int_{\R^n}\int_{\R^n}\frac{(u(x)-u(y))(\phi(x)-\phi(y))}{|x-y|^{n+2s}}\;dy\;dx+\int_\Omega cu\phi\geq (\text{ or } \leq)\int_\Omega g\phi\;dx,$$
    for every nonnegative $\phi\in C^1_c(\Omega)$.
\end{Definition}
\begin{Remark}
It is noteworthy that the above definitions are well-defined due to the fact $u\in\mathcal{L}^s(\R^n)$.
\end{Remark}
\subsection{Weak Harnack estimate}
This section is devoted to the establishment of a weak Harnack-type inequality, which is crucial for our argument. This result slightly extends \cite[Theorem 1.1]{Garain25}, whose proof relies heavily on the proof of that theorem. To avoid repetitions, we just highlight the main changes here. For the homogeneous case, this inequality has already been obtained in \cite[Proposition 3.3]{BV24}. 
  \begin{prop}\label{WHI}
      Let $\Omega\subset \R^n$ be a bounded domain and $u\in H^1_{\mathrm{loc}}(\Omega)\cap \mathcal{L}^s(\R^n)$ be a weak supersolution to the equation 
      \begin{align}\label{AE1}
          -\De u+(-\De)^su+c(x)u=g \text{ in }\Omega,
      \end{align}
where $0\leq c\in L^\infty(\Omega)$ and $g\in L^q(\Omega)$ with $q>\frac{n}{2}$. Suppose that $u\geq 0$ in $B_R(x_0)\subset\Omega$. Then for every $0<r\leq 1$ with $r<\frac{R}{2}$, there exist two constants $\epsilon=\epsilon(n,s,q,\|c\|_{L^\infty(\Omega)})>0$ and $C=C(n,s,q,\|c\|_{L^\infty(\Omega)})$ such that
\begin{align}\label{WH}
    \left( \fint_{B_{\frac{r}{2}}(x_0)} |u|^\epsilon\;dx\right)^\frac{1}{\epsilon}\leq C\left( \essinf_{B_{\frac{r}{2}}(x_0)} u+\left(\frac{r}{R}\right)^2 \mathrm{Tail}(u^-;x_0,R)+r^{2-n}R^{n(1-\frac{1}{q})}\|g\|_{L^q(\Omega)} \right),
\end{align}
where $\mathrm{Tail}(u;x_0,R)$ is given by (\ref{Tail}).
\end{prop}
Before proceeding to the proof of this proposition, we prove several auxiliary results analogous to \cite[Lemma 3.1]{Garain25}, \cite[Lemma 3.2]{Garain25}, and \cite[Lemma 4.1]{Garain25}. Since our arguments follow similar lines, we omit the full proofs and only highlight the effect of the lower-order term.
\begin{Lemma}
    Let $u\in H^1_{\mathrm{loc}}(\Omega)\cap \mathcal{L}^s(\R^n)$ be a supersolution of \eqref{AE1} such that $u\geq 0$ in $B_R(x_0)\subset\Omega$. Suppose that $0\leq \psi\in C^\infty_c(B_r(x_0))$ for some $0<r\leq1$ with $r<R$, and $d>0\;, \delta>0$ are two fixed real numbers. Then, for any $\eta>1$, there exists $C=C(\eta,\|c\|_{L^\infty(\Omega)})>0$ such that 
    \begin{align}\label{Energy}
        \int_{B_r(x_0)}\psi^2|\nabla v^{-\frac{\alpha}{2}}|^2\;dx&+\int_{B_r(x_0)}\int_{B_r(x_0)}\frac{|v^{-\frac{\alpha}{2}}(x)\psi(x)-v^{-\frac{\alpha}{2}}(y)\psi(y)|^2}{|x-y|^{n+2s}}\;dx\;dy\nonumber\\
        &\leq C\alpha^2\Big[\int_{B_r(x_0)}v^{-\alpha}|\nabla\psi|^2\;dx+\|\psi v^\frac{-\alpha}{2}\|_{L^2(B_R(x_0)}^2\nonumber\\
        &+\int_{B_r(x_0)}\int_{B_r(x_0)}\frac{(v^{-\alpha}(x)+v^{-\alpha}(y))(\psi(x)-\psi(y))^2}{|x-y|^{n+2s}}\;dx\;dy\nonumber\\
        &+\Big(\sup_{x\in\mathrm{supp}(\psi)}\int_{\R^n\setminus B_r(x_0)}\frac{dy}{|x-y|^{n+2s}}\nonumber\\
        &+d^{-1}\sup_{x\in\mathrm{supp}(\psi)}\int_{\R^n\setminus B_R(x_0)}\frac{u^{-}(y)}{|x-y|^{n+2s}}\;dy\Big)\int_{B_r(x_0)}v^{-\alpha}\psi^2\;dx\nonumber\\
        &+d^{-1}\|g\|_{L^q(\Omega)}\Big(\delta\|\psi v^\frac{-\alpha}{2}\|_{L^{2^*}(B_R(x_0))}^2+\delta^\frac{-n}{q-n}\|\psi v^\frac{-\alpha}{2}\|_{L^{2}(B_R(x_0))}^2\Big) \Big],
    \end{align}
    where $v=u+d$ and $\alpha=\eta-1>0$.
\end{Lemma}
\begin{proof}
Since $c\geq 0$, $v$ is also a supersolution of equation \eqref{AE1}. To prove the energy estimate, we incorporate $\phi=-\psi^2 v^{-\eta}$ in the weak formulation of \eqref{AE1} to obtain
\begin{align}\label{I1}
    I+J+L\leq K,
\end{align}
where 
\begin{align*}
    &I=\int_\Omega\nabla v\cdot \nabla(-\psi^2 v^{-\eta})\;dx,\\
    &J=\int_{\R^n}\int_{\R^n}\frac{(v(x)-v(y))(\psi^2 v^{-\eta}(y)-\psi^2 v^{-\eta}(x))}{|x-y|^{n+2s}}\;dy\;dx,\\
    &L=-\int_\Omega c(x)\psi^2 v^{1-\eta}\;dx,\\
    &K=-\int_\Omega g\psi^2 v^{-\eta}\;dx.
\end{align*}
Since $c\in L^\infty(\Omega)$, we get $|L|\leq \|c\|_{L^{\infty}(\Omega)}\int_\Omega \psi^2v^{-\alpha}\;dx=\|c\|_{L^\infty(\Omega)}\|\psi v^\frac{-\alpha}{2}\|_{L^2(\Omega)}^2$. Estimating $I,J$ and $K$ in a similar way as in the proof of \cite[Lemma 3.1 (a)]{Garain25} we obtain
\begin{align*}
    &I\geq \frac{\eta}{2\alpha^2}\int_{B_r(x_0)}|\nabla v^\frac{-\alpha}{2}|^2\psi^2\;dx-\frac{2}{\eta}\int_{B_r(x_0)}|\nabla\psi|^2v^{-\alpha}\;dx,\nonumber\\
    &J\geq C\int_{B_r(x_0)}\int_{B_r(x_0)}\frac{|v^{-\frac{\alpha}{2}}(x)\psi(x)-v^{-\frac{\alpha}{2}}(y)\psi(y)|^2}{|x-y|^{n+2s}}\;dx\;dy\nonumber\\
    &-C\int_{B_r(x_0)}\int_{B_r(x_0)}\frac{(v^{-\alpha}(x)+v^{-\alpha}(y))(\psi(x)-\psi(y))^2}{|x-y|^{n+2s}}\;dx\;dy \nonumber\\
    &-2\Big(\sup_{x\in\mathrm{supp}(\psi)}\int_{\R^n\setminus B_r(x_0)}\frac{dy}{|x-y|^{n+2s}}\nonumber\\
        &+d^{-1}\sup_{x\in\mathrm{supp}(\psi)}\int_{\R^n\setminus B_{R}(x_0)}\frac{u^{-}(y)}{|x-y|^{n+2s}}\;dy\Big)\int_{B_r(x_0)}v^{-\alpha}\psi^2\;dx,\nonumber\\
    &|K|\leq d^{-1}\|g\|_{L^q(\Omega)}\Big(\delta\|\psi v^\frac{-\alpha}{2}\|_{L^{2^*}(B_R(x_0))}^2+\delta^\frac{-n}{q-n}\|\psi v^\frac{-\alpha}{2}\|_{L^{2}(B_R(x_0))}^2\Big),
\end{align*}
where the constant $C$ depends only on $\eta$, which is bounded for $\eta>1$. Combining these estimates, the inequality (\ref{I1}) reveals the desired estimate.
\end{proof}
\begin{Lemma}\label{localbound}
    Assume that $u$ satisfies the assumption of Proposition \ref{WHI} and $u\geq 0$ in $B_R(x_0)\subset\Omega$. Suppose that $0<r\leq 1$ with $r<R$ and define $v=u+d$ with
    $$d>\left(\frac{r}{R}\right)^2 \mathrm{Tail}(u^-;x_0,R)+r^{2-n}R^{n(1-\frac{1}{q})}||g\|_{L^q(\Omega)}.$$ Then for every $0<\gamma\leq \sigma'<\sigma\leq 1$ and for every $t>0$, there exist two constants $C=C(n,s,t,\gamma,\|c\|_{L^\infty(\Omega)})$ and $\theta=\theta(n,s)>0$ such that 
    $$\esssup_{B_{\sigma'r}}v^{-1}\leq \Big(\frac{C}{(\sigma-\sigma')^\theta}\Big)^\frac{1}{t}\Big(\fint_{B_{\sigma r}}v^{-t}\;dx\Big)^\frac{1}{t}.$$
   
\end{Lemma}
\begin{proof}
Let us define $\sigma_i=\sigma-(\sigma-\sigma')(1-\frac{1}{\chi^i})$, where $$\chi=\begin{cases}
    \frac{n}{n-2}&\text{ if }n>2,\\
    2&\text{ if }n=2.
\end{cases}$$
It is clear that $\sigma_i$ decreases to $\sigma'$. We further denote $r_i=\sigma_i r$ and $B_i:=B_{r_i}(x_0)$. For each \textcolor{black}{$i\in \mathbb{N}\cup\{0\}$}, we choose $\psi_i\in C^\infty_c(B_i)$ such that $0\leq \psi_i\leq 1$, $\psi_i=1$ in $B_{i+1}$, and $|\nabla\psi_i|\leq \frac{C\chi^i}{(\sigma-\sigma')r_i}$. Furthermore, $\mathrm{dist}(\partial B_i,\mathrm{supp}(\psi_i))\geq \frac{r_i}{2^{i+1}}$. To prove this result, we adopt Moser's iteration technique. To this aim, we estimate
\begin{align}
    \fint_{B_{i+1}}v^{-\alpha\chi}\;dx&\leq \frac{1}{|B_{i+1}|}\int_{B_{i}}|\psi_i v^\frac{-\alpha}{2}|^{2^*}\;dx\leq \frac{C}{|B_{i+1}|}\left(\int_{B_i}|\nabla(\psi_iv^\frac{-\alpha}{2})|^2\;dx\right)^\chi\nonumber\\
    &\leq \frac{C}{|B_{i+1}|}\Big[ \underbrace{\left(\int_{B_i}|\nabla v^\frac{-\alpha}{2}|^2\psi_i^2\;dx\right)^\chi}_{I}+ \underbrace{\left(\int_{B_i}v^{-\alpha}|\nabla\psi_i|^2\;dx\right)^\chi}_{J}\Big]
\end{align}
where $2^*=2\chi$ and $C=C(n)>0$. Here, $\alpha=\eta-1$ with $\eta>1$. Now we estimate $I$ and $J$ separately.\\
\textbf{Estimate of I:} Since $\eta>1$, using the estimate \eqref{Energy}, we have
\begin{align}\label{Ma0}
    I\leq C\alpha^{2\chi}(I_1^\chi+I_2^\chi+I_3^\chi+I_4^\chi+I_5^\chi),
\end{align}
where 
\begin{align*}
    &I_1=\int_{B_{i}}v^{-\alpha}|\nabla\psi|^2\;dx,\; I_2=\int_{B_{i}}\psi v^{-\alpha}\;dx\\
    &I_3= \int_{B_{i}}\int_{B_{i}}\frac{(v^{-\alpha}(x)+v^{-\alpha}(y))(\psi(x)-\psi(y))^2}{|x-y|^{n+2s}}\;dx\;dy,\\
    &I_4= \Big(\sup_{x\in\mathrm{supp}(\psi)}\int_{\R^n\setminus B_{i}}\frac{dy}{|x-y|^{n+2s}}\nonumber\\
        &+d^{-1}\sup_{x\in\mathrm{supp}(\psi)}\int_{\R^n\setminus B_{i}}\frac{u^{-}(y)}{|x-y|^{n+2s}}\;dy\Big)\int_{B_{i}}v^{-\alpha}\psi^2\;dx\\
        &I_5= d^{-1}\|g\|_{L^q(\Omega)}\Big(\delta\|\psi v^\frac{-\alpha}{2}\|_{L^{2^*}(B_{i})}^2+\delta^\frac{-n}{q-n}\|\psi v^\frac{-\alpha}{2}\|_{L^{2}(B_{i})}^2\Big),
\end{align*}
where $\delta>0$ is a constant (to be determined later). Using the property of $\psi_i$, it is immediate that 
\begin{align}\label{Ma1}
    I_1\leq C\left(\frac{\chi^i}{(\sigma-\sigma')r_i}\right)^2\int_{B_i}v^{-\alpha}\;dx,\quad I_2\leq \int_{B_i}v^{-\alpha}\;dx,
\end{align}
for some constant $C=C(n)>0$. Along the lines of proof of \cite[inequality (4.8)-(4.9), page 16]{Garain25}, we obtain
\begin{align}
    &I_3\leq C\left(\frac{\chi^i}{(\sigma-\sigma')r_i}\right)^2\int_{B_i}v^{-\alpha}\;dx\\
   \text{ and } &I_4\leq C\frac{2^{i(n+2s)}}{(\sigma-\sigma')^2r_i^2}\int_{B_i}v^{-\alpha}\;dx.
\end{align}
We set $\delta=\delta_0r_i^{2q-n}R^{(n-2q)(1-\frac{1}{q})}$, where the constant $\delta_0>0$ to be determined later. Using the fact $d>r^{2-n}R^{n(1-\frac{1}{q})}\|g\|_{L^q(\Omega)}$, we deduce
\begin{align}\label{Ma2}
    I_5&\leq \delta_0\Big(\frac{r_i}{R}\Big)^{2(q-1)}\|\psi_iv^\frac{-\alpha}{2}\|_{L^{2^*}(B_i)}^2+\delta_0^\frac{n}{n-2q}r_i^{-2}\|\psi_iv^\frac{-\alpha}{2}\|_{L^{2}(B_i)}^2,\nonumber\\
    &\leq \delta_0\|\psi_iv^\frac{-\alpha}{2}\|_{L^{2^*}(B_i)}^2+\delta_0^\frac{n}{n-2q}r_i^{-2}\|\psi_iv^\frac{-\alpha}{2}\|_{L^{2}(B_i)}^2.
\end{align}
Combining (\ref{Ma1})-(\ref{Ma2}), and using the facts $2i\leq 2i(n+2s+2),\;i(n+2s)\leq 2i(n+2s+2)$, inequality (\ref{Ma0}) yields
\begin{align}\label{Ma3}
    I&\leq C\alpha^\chi\Big( \big(\frac{\chi^i}{(\sigma-\sigma')r_i}\big)^{2\chi}+1+ \big( \frac{2^{i(n+2s)}}{(\sigma-\sigma')^2r_i^2}\big)^\chi \Big) \Big(\int_{B_i}v^{-\alpha}\;dx\Big)^\chi\nonumber\\
    &+C\alpha^{2\chi}\Big(\delta_0^\chi\|\psi_iv^\frac{-\alpha}{2}\|_{L^{2^*}(B_i)}^{2^*}+\delta_0^\frac{n\chi}{n-2q}r_i^{-2\chi}\|\psi_iv^\frac{-\alpha}{2}\|_{L^{2}(B_i)}^{2^*}\Big)\nonumber\\
    &\leq C\alpha^{2\chi}\left(\frac{(2\chi)^{2i(n+2s+2)}}{(\sigma-\sigma')^2r_i^2}\int_{B_i}v^{-\alpha}\;dx\right)^\chi\nonumber\\
    &+C\alpha^{2\chi}\Big(\delta_0^\chi\|\psi_iv^\frac{-\alpha}{2}\|_{L^{2^*}(B_i)}^{2^*}+\delta_0^\frac{n\chi}{n-2q}r_i^{-2\chi}\|\psi_iv^\frac{-\alpha}{2}\|_{L^{2}(B_i)}^{2^*}\Big),
\end{align}
where $C=C(n,s,\gamma, q,\|c\|_{L^\infty(\Omega)})>0$. We Choose $\delta_0>0$ such that $C\alpha^{2\chi}\delta_0^\chi=\frac{1}{2}.$ Thus, (\ref{Ma3}) reveals
\begin{align}\label{Ma4}
    I\leq C\max\{\alpha^{2\chi},\alpha^{2\chi(1+\frac{n}{2q-n})}\}\left(\frac{(2\chi)^{2i(n+2s+2)}}{(\sigma-\sigma')^2r_i^2}\int_{B_i}v^{-\alpha}\;dx\right)^\chi+\frac{1}{2}\|\psi_iv^\frac{-\alpha}{2}\|_{L^{2^*}(B_i)}^{2^*},
\end{align}
where $C=C(n,s,\gamma, q,\|c\|_{L^\infty(\Omega)})>0$. Using the properties of $\psi_i$, we get
\begin{align}\label{Ma5}
    J\leq C\left(\frac{(2\chi)^{2i(n+2s+2)}}{(\sigma-\sigma')^2r_i^2}\int_{B_i}v^{-\alpha}\;dx\right)^\chi.
\end{align}
Combining \eqref{Ma4}, \eqref{Ma5}, we obtain
\begin{align}
    \int_{B_{i}}|\psi_i v^\frac{-\alpha}{2}|^{2^*}\;dx\leq C\Big(\max\{\alpha^{2\chi},\alpha^{2\chi(1+\frac{n}{2q-n})}\}+1\Big)\left(\frac{(2\chi)^{2i(n+2s+2)}}{(\sigma-\sigma')^2r_i^2}\int_{B_i}v^{-\alpha}\;dx\right)^\chi,
\end{align}
where $C=C(n,s,\gamma, q,\|c\|_{L^\infty(\Omega)})>0$. The rest of the proof proceeds verbatim as in \cite[Lemma 4.1]{Garain25}.
\end{proof}

\begin{Lemma}\label{log}
Assume that $u$ satisfies the assumption of Proposition \ref{WHI} and $u\geq 0$ in $B_R(x_0)\subset\Omega$. Suppose that $0<r\leq 1$ with $r<\frac{R}{2}$ and $\psi\in C^\infty_c(B_\frac{3r}{2}(x_0))$ such that $0\leq \psi\leq 1$, and $|\nabla\psi|\leq \frac{C}{r}$ in $B_\frac{3r}{2}(x_0)$ for some $C=C(n)>0$. Define $v=u+d$ for some $d>0$. Then there exists a constant $C=C(n,s,q,\|c\|_{L^\infty(\Omega)})>0$ such that 
\begin{align}\label{AE2}
    \int_{B_\frac{3r}{2}(x_0)}|\nabla (\mathrm{log}v)|^2\psi^2\;dx\leq Cr^n\Big(\frac{1}{r^2}+\frac{1}{dR^2}\mathrm{Tail}(u^-;x_0,R)\Big)+\frac{CR^{n(1-\frac{1}{q})}}{d}\|g\|_{L^q(B_R(x_0))}.
\end{align}  
Furthermore, if $\psi=1$ in $B_r(x_0)$ and 
\begin{align}\label{lbb}
    d>\left(\frac{r}{R}\right)^2 \mathrm{Tail}(u^-,x_0,R)+r^{2-n}R^{n(1-\frac{1}{q})}||g\|_{L^q(\Omega)},
\end{align}
then 
\begin{align}
    \int_{B_{r}(x_0)}|\nabla \mathrm{log}(v)|^2\;dx\leq Cr^{n-2},
\end{align}
where $C=C(n,s,q,\|c\|_{L^\infty(\Omega)})>0$ is a constant.
\end{Lemma}
\begin{proof}
Since $v$ is a supersolution of \eqref{AE1}, incorporating $\phi=\frac{\psi^2}{v}$ in the weak formulation, we deduce
\begin{align}
 \underbrace{\int_\Omega\nabla v\cdot\nabla (\psi^2v^{-1})  \;dx}_{I}&+\underbrace{\int_{\R^n}\int_{\R^n}\frac{(v(x)-v(y))(\psi^2v^{-1}(x)-\psi^2v^{-1}(y))}{|x-y|^{n+2s}}\;dx\;dy}_{J}\nonumber\\
 &+\underbrace{\int_\Omega c(x)\psi^2\;dx}_{L}\geq \underbrace{\int_{\Omega}g\psi^2v^{-1}\;dx}_{K}. 
\end{align}
Since $r<1$, we have $L\leq \|c\|_{L^\infty(\Omega)}|B_{\frac{3r}{2}}|\leq Cr^n\leq Cr^{n-2},$ for some $C=C(n,\|c\|_{L^\infty(\Omega)})>0$. Since $r<\frac{R}{2}$, along the lines of proof of \cite[inequality (3.13)-(3.15), page 13]{Garain25}, we deduce
\begin{align*}
  &I\leq -\frac{1}{2}\int_{B_{\frac{3r}{2}}(x_0)} |\nabla \mathrm{log}(v)|^2\psi^2+Cr^{n-2},\;|K|\leq Cd^{-1}R^{n(1-\frac{1}{q})}\|g\|_{L^q(B_R)},\\
  &J\leq -\frac{1}{C}\int_{B_{2r}(x_0)}\int_{B_{2r}(x_0)}\Big|\mathrm{log}(\frac{v(x)}{v(y)})\Big|^2\psi^2\;dy\;dx+\frac{Cr^n}{dR^2}\mathrm{Tail}(u^-;x_0,R)+Cr^{n-2},
\end{align*}
for some constant $C>0$, depending only on $n$ and $s$. Combining all these estimations, we obtain our desired result.

Since $\psi=1$ in $B_{r}(x_0)$ and (\ref{lbb}) hold, (\ref{AE2}) yields
\begin{align*}
    \int_{B_r(x_0)}|\nabla (\mathrm{log}v)|^2\;dx&\leq Cr^{n-2}+\frac{Cr^{n-2}}{d}\Big(\frac{r^2}{R^2}\mathrm{Tail}(u^-,x_0,R)+r^{2-n}R^{n(1-\frac{1}{q})}\|g\|_{L^q(B_R(x_0))}\Big)\nonumber\\
    &\leq Cr^{n-2},
\end{align*}
for some constant $C>0$.
\end{proof}
\textbf{Proof of Proposition \ref{WHI}:} Using Lemma \ref{localbound}, Lemma \ref{log} together with the John-Nirenberg inequality, along the lines of the proof of Step-1 and Step-2 in \cite[Theorem 1.1]{Garain25}, we obtain our result.\\
The next corollary follows immediately from the proposition \ref{WHI}.
\begin{Corollary}
    Let $u\in H^1_{\mathrm{loc}}(\Omega)\cap \mathcal{L}^s(\R^n)$ be a non-negative weak supersolution to the equation 
    \begin{align*}
          -\De u+(-\De)^su+c(x)u=g \text{ in }\Omega,
      \end{align*}
where $0\leq c\in L^\infty(\Omega)$ and $g\in L^n(\Omega)$. Then, for every $0<r\leq 1$ with $B_{2r}(x_0)\subset \Omega$, there exist two constants $\epsilon=\epsilon(n,s,\|c\|_{L^\infty(\Omega)})>0$ and $C=C(n,s,\|c\|_{L^\infty(\Omega)})$ such that
$$\left( \fint_{B_{r}(x_0)} |u|^\epsilon\;dx\right)^\frac{1}{\epsilon}\leq C\left( \essinf_{B_{r}(x_0)} u+r\|g\|_{L^n(\Omega)} \right).$$
\end{Corollary}
\begin{proof}
    Incorporating $q=n$ and $u^-=0$ in \eqref{WH}, we obtain our desired result.
\end{proof}
\subsection{Auxiliary result}
Before we move on to our main auxiliary results, we define two functions $\xi$ and $\eta$ from $\R$ to $\R$ as follows. For $p,\,q\geq 1$,
\begin{align}\label{Babu1}
     \xi(t)=\frac{q^2}{p}\begin{cases}
       0 &\text{ if } t\leq l,\\
       t^p-l^p &\text{ if }l\leq t\leq r,\\
       p r^{p-1}(t-r)+r^p-l^p &\text{ if }t\geq r,
    \end{cases}
\end{align}
and 
\begin{align}\label{Babu2}
     \eta(t)=\begin{cases}
       l^q &\text{ if } t\leq l,\\
       t^q &\text{ if } l\leq t\leq r,\\
       q r^{q-1}(t-r)+r^q &\text{ if }t\geq r.
    \end{cases}
\end{align}
These functions enjoy the following.
\begin{Lemma}\label{Sa2}
    Suppose that $p,\,q\geq 1$ are two real numbers such that $2(q-1)=(p-1)$. Then for all $t$, the following holds:
    \begin{itemize}
        \item[(a)] $\xi'(t)=\eta'(t)^2$,
        \item[(b)] $t\xi(t)\leq q\eta^2(t)$.
    \end{itemize}     
\end{Lemma}
\begin{proof}
    The proof is straightforward. We omit it here.
\end{proof}
The next lemma is an analogue of Alexadroff-Bakelman-Pucci type inequality, which is crucial for our argument.
\begin{Lemma}\label{SML}
     Let $\Omega\subset\R^n$ be a bounded Lipschitz domain. Suppose that $u\in C^1 (\overline{\Omega})\textcolor{black}{\cap \mathcal{L}^s(\R^n)}$ is a weak subsolution of the equation
\begin{align}\label{Sa}
    \begin{cases}
        -\De u+(-\De)^su+c(x)u=g \text{ in }\Omega,\\
        u\leq 0\text{ on }\textcolor{black}{\R^n\setminus\Omega},
    \end{cases}
\end{align}
where $c,\, g\in L^\infty(\Omega).$ Then, there exist two constants $\delta=\delta(n,\|c\|_{L^\infty(\Omega)})>0$ and $\mathcal{K}=\mathcal{K}(n,\|c\|_{L^\infty(\Omega)})>0$
such that whenever $|\Omega|\leq\delta$, we have $$\|u^+\|_{L^\infty(\Omega)}\leq \mathcal{K} \|g\|_{L^\infty(\Omega)}.$$    
\end{Lemma}
\begin{proof}
Without loss of generality, we assume that $|\Omega|\leq 1$ and $n>2$. We define $w:=u^++k$, where $k=\|g\|_{L^\infty(\Omega)}$. We assume $l>k$ in the aforementioned definition of $\xi$ and $\eta$. Incorporating $\phi=\xi(w)$ in the weak formulation of \eqref{Sa}, we obtain 
\begin{align*}
    \int_\Omega \xi'(w)\nabla u\cdot\nabla w\;dx+&\int_{\R^n}\int_{\R^n}\frac{(u(x)-u(y))(\xi(w)(x)-\xi(w)(y))}{|x-y|^{n+2s}}\;dy\;dx\nonumber\\
    &+\int_{\Omega}c(x)u(x)\xi(w)(x)\;dx\leq\int_{\Omega}g\xi(w)\;dx.
\end{align*}
Using the monotonicity property of $\xi$ and the fact that $ u\geq 0$ in $\mathrm{supp}(\xi(w))$, we obtain
\begin{align*}
    \int_\Omega \xi'(w)|\nabla w|^2\;dx+\int_{\Omega}c(x)u^+\xi(w)\;dx\leq\int_{\Omega}g\xi(w)\;dx.
\end{align*}
Utilizing Lemma \ref{Sa2}, one has
\begin{align}\label{Sa3}
     \int_\Omega |\nabla\eta(w)|^2\;dx\leq\int_{\Omega}(g-cu^+)\xi(w)\;dx&\leq (1+\|c\|_{L^\infty(\Omega)})\int_{\Omega}(k+u^+)\xi(w)\;dx \nonumber\\
     &\leq (1+\|c\|_{L^\infty(\Omega)})\int_\Omega w\xi(w)\;dx\nonumber\\
     &\leq q(1+\|c\|_{L^\infty(\Omega)})\int_\Omega |\eta(w)|^2\;dx.
\end{align}
Denote $v:=\eta(w)$. Since $q>1$, \eqref{Sa3} implies
\begin{align}\label{Sa4}
    \int_\Omega |\nabla v|^2\;dx\leq (1+\|c\|_{L^\infty(\Omega)})q^2\int_\Omega |v|^2\;dx.
\end{align}
Now, Using (\ref{Sa4}), Sobolev inequality and the fact $v>l^q$, we have
\begin{align*}
    \|v\|_{L^{2^*(\Omega)}}\leq \|v-l^q\|_{L^{2^*(\Omega)}}+\|l^q\|_{L^{2^*(\Omega)}}&\leq C\|\nabla v\|_{L^2(\Omega)}+\|l^q\|_{L^{2(\Omega)}}|\Omega|^{-\frac{1}{n}}\nonumber\\
    &\leq |\Omega|^{-\frac{1}{n}}\Big(Cq\|v\|_{L^2(\Omega)}+\|l^q\|_{L^2(\Omega)}\Big)\nonumber\\
    &\leq Cq|\Omega|^{-\frac{1}{n}}\|v\|_{L^2(\Omega)},
\end{align*}
where $C=C(n,\|c\|_{L^\infty(\Omega)})>1$ and $2^*:=\frac{2n}{n-2}$. Since $v\leq w^q+l^q-k^q$, we get
\begin{align*}
    \Big(\int_{l\leq w\leq r}v^{2^*}\;dx\Big)^\frac{1}{2^*}\leq Cq|\Omega|^{-\frac{1}{n}}\left(\|w^q\|_{L^2(\Omega)}+(l^q-k^q)|\Omega|^\frac{1}{2}\right).
\end{align*}
Taking $l\to k$ and $r\to\infty$, we deduce
\begin{align*}
    \|w\|_{L^{2^*q}(\Omega)}^q\leq Cq|\Omega|^{-\frac{1}{n}}\|w\|_{L^{2q}(\Omega)}^q.
\end{align*}
This inequality yields that for $\chi=\frac{n}{n-2}$, one has
\begin{align*}
    \|w\|_{L^{2q\chi}(\Omega)}\leq (Cq)^\frac{1}{q}|\Omega|^{-\frac{1}{nq}}\|w\|_{L^{2q}(\Omega)}.
\end{align*}
 Iterating this process, for every $i\in\mathbb{N}$ we obtain
\begin{align*}
    \|w\|_{L^{2\chi^i}(\Omega)}\leq (C|\Omega|^{-\frac{1}{n}})^{1+\frac{1}{\chi}+..+\frac{1}{\chi^{i-1}}}(\prod_{j=1}^{i-1} \chi^\frac{j}{\chi^j})\|w\|_{L^{2}(\Omega)}\leq C|\Omega|^{-\frac{1}{2}}\|w\|_{L^{2}(\Omega)},
\end{align*}
for some constant $C=C(n,\|c\|_{L^\infty(\Omega)})>0$. Taking $i\to\infty$, we deduce 
\begin{align*}
    \|w\|_{L^\infty(\Omega)}\leq C|\Omega|^{-\frac{1}{2}}\|w\|_{L^{2}(\Omega)}.
\end{align*}
Since $w=u^++k$, the above inequality reveals 
\begin{align}\label{Sa5}
    \|w\|_{L^\infty(\Omega)}\leq C|\Omega|^{-\frac{1}{2}}\|u^+\|_{L^2(\Omega)}+Ck,
\end{align}
for some constant $C=C(n,\|c\|_{L^\infty(\Omega)})>0$.
Moreover, putting $u^+\in H^1_0(\Omega)$ in the weak formulation of \eqref{Sa} and utilizing the Poincar\'{e} inequality, one has
\begin{align}\label{Sa6}
    \|\nabla u^+\|_{L^2(\Omega)}^2\leq \int_\Omega (g-cu)u^+\;dx\leq (1+\|c\|_{L^\infty(\Omega)})\int_\Omega wu^+\;dx\leq C\|w\|_{L^\infty(\Omega)}^2,
\end{align}
Combining \eqref{Sa5}, (\ref{Sa6}) together with the poincar\'{e} inequality, we get
\begin{align}\label{Sa7}
    \|w\|_{L^\infty(\Omega)}\leq C|\Omega|^{\frac{n-1}{2}}\|w\|_{L^\infty(\Omega)}+Ck,
\end{align}
where $C=C(n,\|c\|_{L^\infty(\Omega)})>0$ is a positive constant. Whenever $C|\Omega|^\frac{n-1}{2}\leq \frac{1}{2}$ i.e., $|\Omega|\leq \Big(\frac{1}{2C}\Big)^\frac{2}{n-1}(:=\delta)$, we have $$\|u^+\|_{L^\infty(\Omega)}\leq \|w\|_{L^\infty(\Omega)}\leq 2Ck=2C\|g\|_{L^\infty(\Omega)},$$
for some positive constant $C=C(n,\|c\|_{L^\infty(\Omega)}).$
This completes the proof.
\end{proof}

\section{Preliminaries for Theorem \ref{T1}}
We begin with defining some notation. For $\lambda\in\R$, we define 
\begin{align}\label{M}
    \Sigma_\lambda:=\begin{cases}
        \{x=(x_1,x_2,...,x_n)\in \R^{n}: x_1>\lambda\}, &\text{ if }\lambda>0,\\
        \{x=(x_1,x_2,...,x_n)\in \R^{n}: x_1<\lambda\}, &\text{ if }\lambda<0,
    \end{cases}
\end{align}
and $\Omega_\lambda:=\Sigma_\lambda\cap\Omega.$ Furthermore, we define $x_\lambda=R_\lambda x:=(2\lambda-x_1,x_2,x_3,...,x_n)$, which is the reflection of $x$ through the hyperplane $T_\lambda:=\{x_1=\lambda\}$ and $R_\lambda\Omega_\lambda:=\{x_\lambda: x\in\Omega_\lambda\}$. Define $u_\lambda:\R^n\to \R$ by $u_\lambda(x)=u(x_\lambda)$. Let $a=\sup_{x\in\Omega} x\cdot e_1$.

Since $\Gamma\subset\Omega$ is a closed set such that $\mathrm{cap}_2(\Gamma)=0$, one has $\mathrm{cap}_2(R_\lambda\Gamma)=0$, where $R_\lambda\Gamma$ denotes the reflection of $\Gamma$ with respect to the hyperplane $T_\lambda$ (in general, for any $A\subset \R^n$, $R_\lambda A$ denotes the reflection of $A$ with respect to the hyperplane $T_\lambda$). We also have $\mathcal{H}^{n-2}(T_\lambda\cap \partial\Omega)<\infty$, and hence $\mathrm{cap}_2(T_\lambda\cap\partial\Omega)=0$. Due to these facts, for small $\epsilon_1,\epsilon_2>0$, there exist $\psi_1\in C^\infty_c(B_{\epsilon_1}(R_\lambda\Gamma))$ and $\psi_2\in C^\infty_c(B_{\epsilon_2}(T_\lambda\cap\partial\Omega))$ such that $ \psi_1$ and $\psi_2$ are greater than or equal to $1$ in $B_{\delta_1}(R_\lambda\Gamma)$ and $B_{\delta_2}(T_\lambda\cap\partial\Omega)$, respectively, and
\begin{align}\label{L1}
        \int_{B_{\epsilon_1}(R_\lambda\Gamma)}|\nabla \psi_1|^2\,dx\leq C\epsilon_1,\int_{B_{\epsilon_2}(T_\lambda\cap\partial\Omega)}|\nabla \psi_2|^2\,dx\leq C\epsilon_2,
\end{align}
for some $\delta_i<\epsilon_i$ ($1\leq i\leq 2$) and for some constant $C>0$, independent of $\epsilon_1$ and $\epsilon_2$. Without loss of generality, we assume that $B_{\epsilon_2}(T_\lambda\cap\partial\Omega)$ is symmetric about the plane $T_\lambda$ and $\psi_2(x)=\psi_2(x_\lambda)$ (otherwise, we take the map $x\to\psi_{\epsilon_2}(x)+\psi_{\epsilon_2}(x_\lambda)$).\\
Let us define two maps $\phi_1:\Sigma_\lambda\to \R$ and $\phi_2:\R^n\to \R$ by $$\phi_i:=S\circ\psi_i,$$
where $$S(t)=\begin{cases}
    1,&\text{ if }t\leq 0,\\
    -2t+1,&\text{ if }0\leq t\leq\frac{1}{2},\\
    0,&\text{ if }t\geq \frac{1}{2}.
\end{cases}$$
We extend $\phi_1$ on $\R^n$ by $\phi_1(x)=\phi_1(x_\lambda)$ for $x\in\R^n\setminus\Sigma_\lambda$.
Clearly, $\phi_1,\phi_2\in C^{0,1}(\R^n,[0,1])$; using (\ref{L1}) we have
\begin{align}\label{Ranjit1}
       \int_{B_{\epsilon_1}(R_\lambda\Gamma)}|\nabla \phi_1|^2\,dx\leq 4C\epsilon_1,\text{ and }\int_{B_{\epsilon_2}(T_\lambda\cap\partial\Omega)}|\nabla \phi_2|^2\,dx\leq 4C\epsilon_2.
\end{align}
The following Lemma is very crucial for our argument. 
\begin{Lemma}\label{MLemma}
    For $\lambda\in\R$, define $w_\lambda:\R^n\to\R$ by 
    \begin{align}\label{IF}
        w_\lambda(x):=\begin{cases}
    (u-u_\lambda)^+, &\text{ if }x\in\Sigma_\lambda,\\
    -(u-u_\lambda)^-,&\text{ if }x\in \R^n\setminus\Sigma_\lambda.
\end{cases}
    \end{align}
Then the support of the function $w_\lambda$ is contained in $\overline{\Omega_\lambda\cup R_\lambda\Omega_\lambda}$. Furthermore, $w_\lambda\in H^1_0(\Omega_\lambda\cup R_\lambda\Omega_\lambda)$.
\end{Lemma}
\begin{proof} For the first part, it is enough to show that $w_\lambda\equiv 0$ in $(\Omega_\lambda\cup R_\lambda\Omega_\lambda)^c$. Suppose $x\in \Omega_\lambda^c\cap\Sigma_\lambda$ then $u(x)=0$ and hence $w_\lambda(x)=0$. If $x\in (R_\lambda\Omega_\lambda)^c\cap (\R^n\setminus\Sigma_\lambda)$, then $x_\lambda\in \Omega_\lambda^c \cap \Sigma_\lambda$ and hence $u(x_\lambda)=0$. Thus, $w_\lambda(x)=0$ and thereby the first part is proved. \\
The rest of the proof aims to show that $w_\lambda\in H^1_0(\Omega_\lambda\cup R_\lambda\Omega_\lambda)$. We divided the proof into two steps, where in the first step, we define an admissible test function that will be used in the second step to complete the lemma.\\
\textbf{Step-I:} In this step, we will show that for every $\lambda\in\R$, the function
\begin{align*}
    \phi_\lambda:=w_\lambda\phi_1^2\phi_2^2 =\begin{cases}
        (u-u_\lambda)^+\phi_1^2\phi_2^2, &\text{ in }\Sigma_\lambda,\\
        -(u-u_\lambda)^-\phi_1^2\phi_2^2, &\text{ in }\R^n\setminus\Sigma_\lambda,\\
    \end{cases}
\end{align*}
is Lipschitz continuous and compactly supported in $(\Omega_\lambda\cup R_\lambda\Omega_\lambda)\setminus (\Gamma\cup R_\lambda\Gamma)$. 
Since $u\in C(\overline{\Omega}\setminus\Gamma)$ and $u=0$ on $\partial\Omega$, for every $x\in\partial(\Omega_\lambda\cup R_\lambda\Omega_\lambda)$ there exists a small ball $B_r(x)$ such that $\phi_\lambda=0$ in $B_r(x)$, where the definition of $\phi_2$ has also been used. Moreover, using the definition of $\phi_1$, there exists a neighborhood $U$ of $\Gamma\cup R_\lambda\Gamma$ such that $\phi_\lambda=0$ in $U$. Therefore, $\mathrm{supp}(\phi_\lambda)\Subset (\Omega_\lambda\cup R_\lambda\Omega_\lambda)\setminus (\Gamma\cup R_\lambda\Gamma)$.

Now, we will show that $\phi_\lambda\chi_{\Sigma_\lambda}\in C^{0,1}_c(\R^n)$. It is enough to show that for every $x\in \R^n$, there exists a ball $B_r(x)$ such that $\phi_\lambda \chi_{\Sigma_\lambda}\in C^{0,1}(B_r(x))$. It is immediate that if $x\in\R^n\setminus\overline{\Omega}_\lambda$, then such a ball can be easily found. If $x\in \partial\Omega\cap \overline{\Sigma}_\lambda$, then utilizing the boundary condition and recalling the definition of $\phi_2$, we can find such a small ball. Suppose $x\in \Omega_\lambda\setminus R_\lambda\Gamma$, then there exists a ball $B_r(x)\subset\Omega_\lambda$ such that $u, u_\lambda\in C^1(B_r(x))$ (see, \cite{ValdinociMathZ22}) and therefore $\phi_\lambda\in C^{0,1}(B_r(x))$. When $x\in T_\lambda\cap\overline{\Omega_\lambda}$, there exists a ball $B_r(x)$ such that $B_r(x)\cap\partial\Omega=\emptyset$, $B_r(x)\cap R_\lambda\Gamma=\emptyset$ and $u,u_\lambda\in C^1(B_r(x))$. Therefore, $u,\,u_\lambda\in C^1(B_r(x)\cap \overline{\Omega}_\lambda)$. Thus, $\phi_\lambda \chi_{\Sigma_\lambda}\in C^{0,1}(B_r(x))$ and thereby $\phi_\lambda\chi_{\Sigma_\lambda}\in C^{0,1}_c(\R^n)$. Arguing in a similar way, we obtain $\phi_\lambda\chi_{\R^n\setminus\Sigma_\lambda}\in C^{0,1}_c(\R^n)$. Consequently, $\phi_\lambda=\phi_\lambda\chi_{\Sigma_\lambda}+\phi_\lambda\chi_{\R^n\setminus\Sigma_\lambda}\in C^{0,1}_c(\R^n)$. Moreover, $$\nabla\phi_\lambda=\phi_1^2\phi_2^2\nabla w_\lambda+2\phi_1\phi_2^2w_\lambda\nabla\phi_1+2\phi_1^2\phi_2w_\lambda\nabla\phi_2\in L^2(\Omega).$$
\textbf{Step-II:} Here, we prove that $w_\lambda\in H^1_0(\Omega_\lambda\cup R_\lambda\Omega_\lambda)$. The step-I allowed us to incorporate $\phi_\lambda$ in the weak formulation of equation \eqref{ME1}. Thus, we have
\begin{align}\label{Laxmi1}
    \int_\Omega \nabla u\cdot\nabla\phi_\lambda\;dx+\int_{\R^n}\int_{\R^n}\frac{(u(x)-u(y))(\phi_\lambda(x)-\phi_\lambda(y))}{|x-y|^{n+2s}}\;dy\;dx=\int_\Omega f(x,u)\phi_\lambda\;dx
\end{align}
and
\begin{align}\label{Laxmi2}
    \int_\Omega \nabla u_\lambda\cdot\nabla\phi_\lambda\;dx+\int_{\R^n}\int_{\R^n}\frac{(u_\lambda(x)-u_\lambda(y))(\phi_\lambda(x)-\phi_\lambda(y))}{|x-y|^{n+2s}}\;dy\;dx=\int_\Omega f(x_\lambda,u_\lambda)\phi_\lambda\;dx.
\end{align}
Subtracting (\ref{Laxmi2}) from (\ref{Laxmi1}), we get
\begin{align}\label{Laxmi3}
    \int_\Omega \nabla (u-u_\lambda)\cdot\nabla\phi_\lambda\;dx&+\underbrace{\int_{\R^n}\int_{\R^n}\frac{((u-u_\lambda)(x)-(u-u_\lambda)(y))(\phi_\lambda(x)-\phi_\lambda(y))}{|x-y|^{n+2s}}\;dy\;dx}_{I=I(\epsilon_1,\epsilon_1,\epsilon_2)}\nonumber\\
    &=\int_\Omega (f(x,u)-f(x_\lambda,u_\lambda))\phi_\lambda\;dx.
\end{align}
Recalling the definition of $\phi_\lambda$, and using (\ref{mp}) and ($\mathscr{F}_1$), we deduce that (\ref{Laxmi3}) yields
\begin{align}\label{Laxmi5}
    &\int_{\Omega_\lambda}|\nabla(u-u_\lambda)^+|^2\phi_1^2\phi_2^2\;dx+\int_{\Omega_\lambda}(u-u_\lambda)^+\nabla(u-u_\lambda)^+\cdot\nabla(\phi_1^2\phi_2^2)\;dx\nonumber\\
    &+\int_{R_\lambda\Omega_\lambda}|\nabla(u-u_\lambda)^-|^2\phi_1^2\phi_2^2\;dx+\int_{R_\lambda\Omega_\lambda}(u-u_\lambda)^-\nabla(u-u_\lambda)^-\cdot\nabla(\phi_1^2\phi_2^2)\;dx+I\nonumber\\
    &= \int_{\Omega_\lambda} (f(x,u)-f(x_\lambda,u_\lambda))(u-u_\lambda)^+\phi_1^2\phi_2^2\;dx+\int_{R_\lambda\Omega_\lambda} (f(x_\lambda,u_\lambda)-f(x,u))(u-u_\lambda)^-\phi_1^2\phi_2^2\;dx\nonumber\\
    &\leq \int_{\Omega_\lambda} (f(x,u)-f(x,u_\lambda))(u-u_\lambda)^+\phi_1^2\phi_2^2\;dx+\int_{R_\lambda\Omega_\lambda} (f(x_\lambda,u_\lambda)-f(x_\lambda,u))(u-u_\lambda)^-\phi_1^2\phi_2^2\;dx\nonumber\\
    &\leq C\int_{\Omega_\lambda} |(u-u_\lambda)^+|^2\phi_1^2\phi_2^2\;dx+C\int_{R_\lambda\Omega_\lambda} |(u-u_\lambda)^-|^2\phi_1^2\phi_2^2\;dx,
\end{align}
where $C$ is independent of $\epsilon_1,\epsilon_2$. In the last inequality, we utilized condition ($\mathscr{F}_1$). The inequality (\ref{Laxmi5}) yields
\begin{align}\label{Laxmi4}
 &\int_{\Omega_\lambda}|\nabla(u-u_\lambda)^+|^2\phi_1^2\phi_2^2\;dx+\int_{R_\lambda\Omega_\lambda}|\nabla(u-u_\lambda)^-|^2\phi_1^2\phi_2^2\;dx+I\nonumber\\
  &\leq C\int_{\Omega_\lambda} |(u-u_\lambda)^+|^2\phi_1^2\phi_2^2\;dx+C\int_{R_\lambda\Omega_\lambda} |(u-u_\lambda)^-|^2\phi_1^2\phi_2^2\;dx\nonumber\\
  &-\int_{\Omega_\lambda}(u-u_\lambda)^+\nabla(u-u_\lambda)^+\cdot\nabla(\phi_1^2\phi_2^2)\;dx-\int_{R_\lambda\Omega_\lambda}(u-u_\lambda)^-\nabla(u-u_\lambda)^-\cdot\nabla(\phi_1^2\phi_2^2)\;dx  \nonumber\\
  &\leq C\int_{\Omega_\lambda} |(u-u_\lambda)^+|^2\;dx+C\int_{R_\lambda\Omega_\lambda} |(u-u_\lambda)^-|^2\;dx+2\int_{\Omega_\lambda}(u-u_\lambda)^+|\nabla(u-u_\lambda)^+||\nabla\phi_1|\phi_1\phi_2^2\;dx\nonumber\\
  &+2\int_{\Omega_\lambda}(u-u_\lambda)^+|\nabla(u-u_\lambda)^+||\nabla\phi_2|\phi_2\phi_1^2\;dx+2\int_{R_\lambda\Omega_\lambda}(u-u_\lambda)^-|\nabla(u-u_\lambda)^-||\nabla\phi_1|\phi_1\phi_2^2\;dx  \nonumber\\
  &+2\int_{R_\lambda\Omega_\lambda}(u-u_\lambda)^-|\nabla(u-u_\lambda)^-||\nabla\phi_2|\phi_2\phi_1^2\;dx\nonumber\\
 &\leq C\int_{\Omega_\lambda} |(u-u_\lambda)^+|^2\;dx+C\int_{R_\lambda\Omega_\lambda} |(u-u_\lambda)^-|^2\;dx+\frac{1}{4}\int_{\Omega_\lambda}|\nabla(u-u_\lambda)^+|^2\phi_1^2\phi_2^2\;dx\nonumber\\
  &+16\int_{\Omega_\lambda}|(u-u_\lambda)^+|^2|\nabla\phi_1|^2\phi_2^2\;dx+\frac{1}{4}\int_{\Omega_\lambda}|\nabla(u-u_\lambda)^+|^2\phi_2^2\phi_1^2\;dx+16\int_{\Omega_\lambda}|(u-u_\lambda)^+|^2|\nabla\phi_2|^2\phi_2^2\;dx\nonumber\\
  &+\frac{1}{4}\int_{R_\lambda\Omega_\lambda}|\nabla(u-u_\lambda)^-|^2\phi_2^2\phi_1^2\;dx+16\int_{R_\lambda\Omega_\lambda}|(u-u_\lambda)^-|^2|\nabla\phi_1|^2\phi_2^2\;dx\nonumber\\ 
  &+\frac{1}{4}\int_{R_\lambda\Omega_\lambda}|\nabla(u-u_\lambda)^-|^2\phi_2^2\phi_1^2\;dx+16\int_{R_\lambda\Omega_\lambda}|(u-u_\lambda)^-|^2|\nabla\phi_2|^2\phi_1^2\;dx.
\end{align}
The last inequality was obtained by using Young's inequality. Using $(u-u_\lambda)^+\leq u$ in $\Omega_\lambda$, $(u-u_\lambda)^-\leq u_\lambda$ in $R_\lambda\Omega_\lambda$ and $u\in L^\infty(\Omega_\lambda)$, the Inequality \eqref{Laxmi4} leads to conclude
\begin{align}\label{ImpI}
    &\int_{\Omega_\lambda}|\nabla(u-u_\lambda)^+|^2\phi_1^2\phi_2^2\;dx+\int_{R_\lambda\Omega_\lambda}|\nabla(u-u_\lambda)^-|^2\phi_1^2\phi_2^2\;dx+2I\nonumber\\
    &\leq C\int_{\Omega_\lambda} |(u-u_\lambda)^+|^2\;dx+C\int_{R_\lambda\Omega_\lambda} |(u-u_\lambda)^-|^2\;dx+32\big(\int_{\Omega_\lambda}|(u-u_\lambda)^+|^2|\nabla\phi_1|^2\phi_2^2\;dx\nonumber\\
    &+\int_{\Omega_\lambda}|(u-u_\lambda)^+|^2|\nabla\phi_2|^2\phi_1^2\;dx+\int_{R_\lambda\Omega_\lambda}|(u-u_\lambda)^-|^2|\nabla\phi_1|^2\phi_2^2\;dx\nonumber\\
    &+\int_{R_\lambda\Omega_\lambda}|(u-u_\lambda)^-|^2|\nabla\phi_2|^2\phi_1^2\;dx\big),
\end{align}
    which yields
    \begin{align}\label{Ranjit2}
    &\int_{\Omega_\lambda}|\nabla(u-u_\lambda)^+|^2\phi_1^2\phi_2^2\;dx+\int_{R_\lambda\Omega_\lambda}|\nabla(u-u_\lambda)^-|^2\phi_1^2\phi_2^2\;dx+2I\nonumber\\
    &\leq C\Big(|\Omega_\lambda|+\int_{B_{\epsilon_1}(R_\lambda\Gamma)}|\nabla\phi_1|^2\;dx+\int_{B_{\epsilon_2}(T_\lambda\cap\partial\Omega)}|\nabla\phi_2|^2\;dx\Big)\nonumber\\
    &\leq C(|\Omega|+\epsilon_1+\epsilon_2),
\end{align}
where the constant $C=C(f,\|u\|_{L^\infty(\Omega_\lambda)})$ is independent of $\epsilon_1$ and $\epsilon_2$. The last inequality is obtained by using \eqref{Ranjit1}.
Now, we claim that
\begin{align}\label{Claim1}
   I=\int_{\R^n}\int_{\R^n}\frac{((u-u_\lambda)(x)-(u-u_\lambda)(y))(\phi_\lambda(x)-\phi_\lambda(y))}{|x-y|^{n+2s}}\;dy\;dx\geq o(\epsilon_1,\epsilon_2).
\end{align}
To this concern, we split the integral as follows
\begin{align*}
   \int_{\R^n}\int_{\R^n}&\frac{((u-u_\lambda)(x)-(u-u_\lambda)(y))(\phi_\lambda(x)-\phi_\lambda(y))}{|x-y|^{n+2s}}\;dy\;dx\nonumber\\
   &=\int_{\R^n}\int_{\R^n}\frac{((u-u_\lambda-w_\lambda)(x)-(u-u_\lambda-w_\lambda)(y))(\phi_\lambda(x)-\phi_\lambda(y))}{|x-y|^{n+2s}}\;dy\;dx\nonumber\\
   &+\int_{\R^n}\int_{\R^n}\frac{(w_\lambda(x)-w_\lambda(y))(\phi_\lambda(x)-\phi_\lambda(y))}{|x-y|^{n+2s}}\;dy\;dx=J_1+J_2.
\end{align*}
Since $\phi_1,\,\phi_2$ are symmetric through the plane $T_\lambda$, using similar argument to \cite[Inequality (3.11), page 949]{Luigi18}, we obtain
\begin{align*}
    J_1=\int_{\R^n}\int_{\R^n}\frac{((u-u_\lambda-w_\lambda)(x)-(u-u_\lambda-w_\lambda)(y))(\phi_\lambda(x)-\phi_\lambda(y))}{|x-y|^{n+2s}}\;dy\;dx\geq 0.
\end{align*} 
In order to prove \eqref{Claim1}, it is enough to show that $J_2\geq o(\epsilon_1,\epsilon_2).$
To this aim, we have
\begin{align}\label{Ra1}
    &\int_{\R^n}\int_{\R^n}\frac{(w_\lambda(x)-w_\lambda(y))^2\phi_1^2(x)\phi_2^2(x)}{|x-y|^{n+2s}}\;dy\;dx=\int_{\R^n}\int_{\R^n}\frac{(w_\lambda(x)-w_\lambda(y))(\phi_\lambda(x)-\phi_\lambda(y))}{|x-y|^{n+2s}}\;dy\;dx\nonumber\\
    &+\underbrace{\int_{\R^n}\int_{\R^n}\frac{(w_\lambda(x)-w_\lambda(y))(\phi_1^2(y)\phi_2^2(y)-\phi_1^2(x)\phi_2^2(x))w_\lambda(y)}{|x-y|^{n+2s}}\;dy\;dx}_{J_3}=J_2+J_3.
\end{align}
Utilizing Young's inequality, we derive
\begin{align*}
    J_3&\leq \alpha \int_{\R^n}\int_{\R^n}\frac{(w_\lambda(x)-w_\lambda(y))^2(\phi_1(y)\phi_2(y)+\phi_1(x)\phi_2(x))^2}{|x-y|^{n+2s}}\;dy\;dx\nonumber\\
    &+C(\alpha)\int_{\R^n}\int_{\R^n}\frac{w_\lambda(y)^2(\phi_1(y)\phi_2(y)-\phi_1(x)\phi_2(x))^2}{|x-y|^{n+2s}}\;dy\;dx\nonumber\\
    &\leq 2\alpha \int_{\R^n}\int_{\R^n}\frac{(w_\lambda(x)-w_\lambda(y))^2(\phi_1^2(y)\phi_2^2(y)+\phi_1^2(x)\phi_2^2(x))}{|x-y|^{n+2s}}\;dy\;dx\nonumber\\
    &+C(\alpha)\|u\|^2_{L^\infty(\Omega_\lambda)}\int_{\R^n}\int_{\R^n}\frac{(\phi_1(y)\phi_2(y)-\phi_1(x)\phi_2(x))^2}{|x-y|^{n+2s}}\;dy\;dx\nonumber\\
    &\leq 2\alpha \int_{\R^n}\int_{\R^n}\frac{(w_\lambda(x)-w_\lambda(y))^2(\phi_1^2(y)\phi_2^2(y)+\phi_1^2(x)\phi_2^2(x))}{|x-y|^{n+2s}}\;dy\;dx\nonumber\\
    &+C(\alpha)\|u\|^2_{L^\infty(\Omega_\lambda)}\int_{\R^n} |\nabla(\phi_1\phi_2)|^2\;dx\nonumber\\
    &\leq 2\alpha \int_{\R^n}\int_{\R^n}\frac{(w_\lambda(x)-w_\lambda(y))^2(\phi_1^2(y)\phi_2^2(y)+\phi_1^2(x)\phi_2^2(x))}{|x-y|^{n+2s}}\;dy\;dx\nonumber\\
    &+C(\alpha)\|u\|^2_{L^\infty(\Omega_\lambda)}\int_{\R^n} (|\nabla \phi_1|^2+|\nabla\phi_2|^2)\;dx\nonumber\\
    &\leq 4\alpha \int_{\R^n}\int_{\R^n}\frac{(w_\lambda(x)-w_\lambda(y))^2\phi_1^2(x)\phi_2^2(x)}{|x-y|^{n+2s}}\;dy\;dx\nonumber\\
    &+C(\alpha)\|u\|^2_{L^\infty(\Omega_\lambda)}\big(\int_{B_{\epsilon_1(R_\lambda\Gamma)}} |\nabla\phi_1|^2\;dx+\int_{B_{\epsilon_2(T_\lambda\cap\partial\Omega)}} |\nabla\phi_2|^2\;dx\big)\nonumber\\
    &\leq 4\alpha \int_{\R^n}\int_{\R^n}\frac{(w_\lambda(x)-w_\lambda(y))^2\phi_1^2(x)\phi_2^2(x)}{|x-y|^{n+2s}}\;dy\;dx\nonumber\\
    &+C(\alpha)\|u\|^2_{L^\infty(\Omega_\lambda)}(\epsilon_1+\epsilon_2).
\end{align*}
By choosing $\alpha=\frac{1}{8}$ in the above inequality and using this in \eqref{Ra1}, we obtain $$0\leq\int_{\R^n}\int_{\R^n}\frac{(w_\lambda(x)-w_\lambda(y))^2\phi_1^2(x)\phi_2^2(x)}{|x-y|^{n+2s}}\;dy\;dx\leq 2J_2+C(\alpha)\|u\|^2_{L^\infty(\Omega_\lambda)}(\epsilon_1+\epsilon_2),$$
which reveals $J_2\geq o(\epsilon_1,\epsilon_2)$. This proves our claim (\ref{Claim1}).
Utilize (\ref{Claim1}) and (\ref{Ranjit2}) to obtain
\begin{align*}
    &\int_{\Omega_\lambda}|\nabla(u-u_\lambda)^+|^2\phi_1^2\phi_2^2\;dx+\int_{R_\lambda\Omega_\lambda}|\nabla(u-u_\lambda)^-|^2\phi_1^2\phi_2^2\;dx\leq C(1+\epsilon_1+\epsilon_2),
\end{align*}
where $C>0$ is a constant independent of $\epsilon_1,\,\epsilon_2$. Taking $\epsilon_1,\,\epsilon_2\to 0$ and using Fatou's lemma, we get
$$\int_{\Omega_\lambda\cup R_\lambda\Omega_\lambda}|\nabla w_\lambda|^2\;dx\leq C.$$ Finally, using Lebesgue's dominated convergence theorem, we deduce that $\phi_\lambda\to w_\lambda$  and $\nabla\phi_\lambda\to \nabla w_\lambda$ in $L^2(\Omega_\lambda\cup R_\lambda\Omega_\lambda)$. Thus, the completeness of the space $H^1_0(\Omega_\lambda\cup R_\lambda\Omega_\lambda)$ ensures $w_\lambda\in H^1_0(\Omega_\lambda\cup R_\lambda\Omega_\lambda)$.
\end{proof}
\begin{Remark}\label{R1}
    Suppose that $U\Subset\Omega_\lambda\setminus R_\lambda\Gamma$ such that $u<u_\lambda$ in $U$. Then, we have $$w_\lambda\in H^1_0((\Omega_\lambda\cup R_\lambda\Omega_\lambda)\setminus (U\cup R_\lambda U)).$$
\end{Remark}
\section{Proof of Theorem \ref{T1}}
Throughout this section, we use the aforementioned notation.\\
\textbf{Proof of Theorem \ref{T1}:} We define $$\Lambda:=\{\lambda>0: u\leq u_\mu \text{ in }\Omega_\mu\setminus R_\mu\Gamma \text{ for all }\mu\in (\lambda,a)\}.$$ The proof of this theorem is divided into three steps.\\
\textbf{Step-I:} Firstly, we will prove that $\Lambda\neq\emptyset$. In this concern, we choose $\lambda>0$ such that $\Omega_\lambda\cap R_\lambda\Gamma=\emptyset$. Arguing as Step-I of the proof of Lemma \ref{MLemma}, we infer that the function $\phi_\lambda=w_\lambda\phi_1^2$ can be incorporated as a test function in the weak formulation of equation (\ref{ME1}). Along the lines of the proof of inequality (\ref{ImpI}), we deduce
\begin{align}\label{Satya1}
    &\int_{\Omega_\lambda}|\nabla(u-u_\lambda)^+|^2\phi_1^2\;dx+\int_{R_\lambda\Omega_\lambda}|\nabla(u-u_\lambda)^-|^2\phi_1^2\;dx+2I\nonumber\\
    &\leq C\Big(\int_{\Omega_\lambda\cup R_\lambda\Omega_\lambda} |w_\lambda|^2\;dx+\int_{\Omega_\lambda}|(u-u_\lambda)^+|^2|\nabla\phi_1|^2\;dx+\int_{R_\lambda\Omega_\lambda}|(u-u_\lambda)^-|^2|\nabla\phi_1|^2\;dx\Big)\nonumber\\
    &\leq C\Big(\int_{\Omega_\lambda\cup R_\lambda\Omega_\lambda} |w_\lambda|^2\;dx+\int_{B_{\epsilon_1}(T_\lambda\cap\partial\Omega)}|\nabla\phi_1|^2\;dx\Big)
\end{align}
where the constant $C$ depends on $f,\,\|u\|_{L^\infty(\Omega_\lambda)}$ and 
$$I=I(\epsilon_1)=\int_{\R^n}\int_{\R^n}\frac{((u-u_\lambda)(x)-(u-u_\lambda)(y))(\phi_\lambda(x)-\phi_\lambda(y))}{|x-y|^{n+2s}}\;dy\;dx.$$
Along the lines of proof of the fact \eqref{Claim1}, we have $I(\epsilon_1)\geq o(\epsilon_1)$. This, together with Fatou's lemma, leads to conclude that 
\begin{align*}
    \int_{\Omega_\lambda\cup R_\lambda\Omega_\lambda}|\nabla w_\lambda|^2\,dx\leq C\int_{\Omega_\lambda\cup R_\lambda\Omega_\lambda}|w_\lambda|^2\;dx.
\end{align*}
Since $w_\lambda\in H^1_0(\Omega_\lambda\cup R_\lambda\Omega_\lambda)$, Utilize the Poincar\'{e} inequality to obtain
\begin{align}\label{PI}
    \int_{\Omega_\lambda\cup R_\lambda\Omega_\lambda}|\nabla w_\lambda|^2\,dx\leq C|\Omega_\lambda\cup R_\lambda\Omega_\lambda|^\frac{2}{n}\int_{\Omega_\lambda\cup R_\lambda\Omega_\lambda}|\nabla w_\lambda|^2\;dx,
\end{align}
where $C=C(n,f,\|u\|_{L^\infty(\Omega)})>0$. Choose $\lambda>0$ such that $C|\Omega_\lambda\cup R_\lambda\Omega_\lambda|^\frac{2}{n}\leq \frac{1}{2}$. Thus, (\ref{PI}) yields 
\begin{align*}
     \int_{\Omega_\lambda\cup R_\lambda\Omega_\lambda}|\nabla w_\lambda|^2\,dx\leq 0,
\end{align*}
which leads us to conclude that
$w_\mu=0$ in $\Omega_\mu\setminus R_\mu\Gamma$ for all $\mu\in (\lambda,a)$. Consequently, $\lambda\in \Lambda$, and hence $\Lambda$ is nonempty.\\
\textbf{Step-II:} This step aims to prove $\lambda^*:=\inf \Lambda=0$. To this end, we suppose that $\lambda^*>0$ and come up with a contradiction. Using the continuity of the map $(\lambda,x)\to u(x_\lambda)$, one has $$u\leq u_{\lambda^*}\text{ in }\Omega_{\lambda^*}\setminus R_{\lambda^*}\Gamma.$$
We claim that $u<u_{\lambda^*}\text{ in }\Omega_{\lambda^*}\setminus R_{\lambda^*}\Gamma.$ If not, then there exists $x_0\in \Omega_{\lambda^*}\setminus R_{\lambda^*}\Gamma$ such that $u(x_0)=u_{\lambda^*}(x_0)$. Moreover, there exists a ball $B_r(x_0)\Subset\Omega_{\lambda^*}\setminus R_{\lambda^*}\Gamma$ and $\eta>0$ such that $u,u_{\lambda^*}\geq \eta>0$ in $B_r(x_0)$ and $$\inf_{B_r(x_0)} (u_{\lambda^*}-u)=0.$$ 
Now we prove that $\Tilde{w}=u_{\lambda^*}-u$ is a supersolution of a linear equation in $B_r(x_0)$. To this aim, we suppose $v\in C_c^\infty(B_r(x_0))$ is a nonnegative function. Incorporating $v$ and $v_{\lambda^*}$ in the weak formulation of (\ref{ME1}), we deduce
\begin{align}\label{LaxmiB1}
    \int_\Omega \nabla u\cdot\nabla v\;dx+\int_{\R^n}\int_{\R^n}\frac{(u(x)-u(y))(v(x)-v(y))}{|x-y|^{n+2s}}\;dy\;dx=\int_\Omega f(x,u)v\;dx
\end{align}
and
\begin{align}\label{LaxmiB2}
    \int_\Omega \nabla u_{\lambda^*}\cdot\nabla v\;dx+\int_{\R^n}\int_{\R^n}\frac{(u_{\lambda^*}(x)-u_{\lambda^*}(y))(v(x)-v(y))}{|x-y|^{n+2s}}\;dy\;dx=\int_\Omega f(x_{\lambda^*},u_{\lambda^*})v\;dx.
\end{align}
Subtracting (\ref{LaxmiB1}) from (\ref{LaxmiB2}), we get
\begin{align}\label{LaxmiB3}
    \int_\Omega \nabla (u_{\lambda^*}-u)\cdot\nabla v\;dx&+\underbrace{\int_{\R^n}\int_{\R^n}\frac{((u_{\lambda^*}-u)(x)-(u_{\lambda^*}-u)(y))(v(x)-v(y))}{|x-y|^{n+2s}}\;dy\;dx}_{I=I(\epsilon_1,\epsilon_1,\epsilon_2)}\nonumber\\
    &=\int_\Omega (f(x_{\lambda^*},u_{\lambda^*})-f(x,u))v\;dx\nonumber\\
    &\geq \int_{B_r(x_0)} (f(x,u_{\lambda^*})-f(x,u))v\;dx.
\end{align}
Use condition ($\mathscr{F}_2$) to obtain
\begin{align*}
    \int_\Omega \nabla (u_{\lambda^*}-u)\cdot\nabla v\;dx&+\int_{\R^n}\int_{\R^n}\frac{((u_{\lambda^*}-u)(x)-(u_{\lambda^*}-u)(y))(v(x)-v(y))}{|x-y|^{n+2s}}\;dy\;dx\nonumber\\
    &\geq -M(r,\eta) \int_{B_r(x_0)} (u_{\lambda^*}-u)v\;dx,
\end{align*}
which implies that $\Tilde{w}=u_{\lambda^*}-u$ is a supersolution of the equation
$$-\De z+(-\De)^sz+Mz=0 \text{ in }B_r(x_0).$$
Thus, \cite[Proposition 3.3]{BV24} ensures $u=u_{\lambda^*}$ in $B_r(x_0)$. Since $\Omega_{\lambda^*}\setminus R_{\lambda^*}\Gamma$ is connected, by the usual covering argument, one has $u=u_{\lambda^*}$ in $\Omega_{\lambda^*}\setminus R_{\lambda^*}\Gamma$, which is a contradiction. Hence, $$u<u_{\lambda^*}\text{ in }\Omega_{\lambda^*}\setminus R_{\lambda^*}\Gamma.$$
Since $|R_{\lambda^*}\Gamma|=0$, we can choose $K\Subset \Omega_{\lambda^*}\setminus R_{\lambda^*}\Gamma$ such that 
$|\Omega_{\lambda^*}\setminus K|\leq \frac{1}{2}(\frac{1}{2C})^\frac{n}{2},$
where $C>0$ is given by (\ref{vai}). Since $u_{\lambda^*}-u>0$ in $\Omega_{\lambda^*}\setminus R_{\lambda^*}\Gamma$, there exists $\eta_1>0$ such that $$u_{\lambda^*}-u\geq 2\eta_1>0\text{ in } K.$$
By continuity, there exists $\lambda^{**}\in (0,\lambda^*)$ such that the following hold:
\begin{enumerate}
    \item[(i)] $|\Omega_{\lambda^{**}}\setminus\Omega_{\lambda^*}|\leq \frac{1}{2}(\frac{1}{2C})^\frac{n}{2}$ with $C>0$ is given in (\ref{vai}),
    \item[(ii)] for every $\mu\in (\lambda^{**},\lambda^*)$ we have $$u_\mu-u>\eta_1>0\text{ in }K.$$
\end{enumerate}
 Note that 
\begin{align}\label{vai2}
    |\Omega_{\lambda^{**}}\setminus K|^\frac{2}{n}\leq (|\Omega_{\lambda^{**}}\setminus\Omega_{\lambda^*}| +|\Omega_{\lambda^*}\setminus K|)^\frac{2}{n}\leq \frac{1}{2C}, 
\end{align}
where $C$ is given in \eqref{vai}. Incorporating $\phi_{\lambda^{**}}=w_{\lambda^{**}}\phi_1^2\phi_2^2$ in the weak formulation of (\ref{ME1}) and along the lines of proof of (\ref{ImpI}) we obtain
\begin{align}
    &\int_{\Omega_{\lambda^{**}}}|\nabla(u-u_{\lambda^{**}})^+|^2\phi_1^2\phi_2^2\;dx+\int_{R_{\lambda^{**}}\Omega_{\lambda^{**}}}|\nabla(u-u_{\lambda^{**}})^-|^2\phi_1^2\phi_2^2\;dx+2I\nonumber\\
    &\leq C\int_{\Omega_{\lambda^{**}\cup R_{\lambda^{**}}\Omega_{\lambda^{**}}}} |w_{\lambda^{**}}|^2\;dx+32\big(\int_{\Omega_{\lambda^{**}}}|(u-u_{\lambda^{**}})^+|^2|\nabla\phi_1|^2\phi_2^2\;dx\nonumber\\
    &+\int_{\Omega_{\lambda^{**}}}|(u-u_{\lambda^{**}})^+|^2|\nabla\phi_2|^2\phi_1^2\;dx+\int_{R_{\lambda^{**}}\Omega_{\lambda^{**}}}|(u-u_{\lambda^{**}})^-|^2|\nabla\phi_1|^2\phi_2^2\;dx\nonumber\\
    &+\int_{R_{\lambda^{**}}\Omega_{\lambda^{**}}}|(u-u_{\lambda^{**}})^-|^2|\nabla\phi_2|^2\phi_1^2\;dx\big)\nonumber\\
    &\leq C\Big( \int_{\Omega_{\lambda^{**}\cup R_{\lambda^{**}}\Omega_{\lambda^{**}}}} |w_{\lambda^{**}}|^2\;dx+\epsilon_1+\epsilon_2\Big),
\end{align}
where the constant $C$ depends on $f\,\mbox{ and }\|u\|_{L^\infty(\Omega_\lambda)}$, and 
$$I=I(\epsilon_1,\epsilon_2)=\int_{\R^n}\int_{\R^n}\frac{((u-u_{\lambda^{**}})(x)-(u-u_{\lambda^{**}})(y))(\phi_{\lambda^{**}}(x)-\phi_{\lambda^{**}}(y))}{|x-y|^{n+2s}}\;dy\;dx.$$
Along the lines of proof of the fact \eqref{Claim1}, we have $I(\epsilon_1,\epsilon_2)\geq o(\epsilon_1,\epsilon_2)$. This, together with Fatou's lemma, leads us to conclude that 
\begin{align}\label{vai}
    \int_{(\Omega_{\lambda^{**}}\cup R_{\lambda^{**}}\Omega_{\lambda^{**}})\setminus (K\cup R_{\lambda^{**}} K)}|\nabla w_{\lambda^{**}}|^2\,dx\leq C\int_{(\Omega_{\lambda^{**}}\cup R_{\lambda^{**}}\Omega_{\lambda^{**}})\setminus (K\cup R_{\lambda^{**}} K)}|w_{\lambda^{**}}|^2\;dx.
\end{align}
Due to Remark \ref{R1}, Poincar\'{e} inequality leads to conclude that
\begin{align*}
     \int_{(\Omega_{\lambda^{**}}\cup R_{\lambda^{**}}\Omega_{\lambda^{**}})\setminus (K\cup R_{\lambda^{**}} K)}|\nabla w_{\lambda^{**}}|^2\,dx\leq C |\Omega_{\lambda^{**}}\setminus K|^\frac{2}{n}\int_{(\Omega_{\lambda^{**}}\cup R_{\lambda^{**}}\Omega_{\lambda^{**}})\setminus (K\cup R_{\lambda^{**}} K)}|\nabla w_{\lambda^{**}}|^2\;dx,
\end{align*}
for some constant $C=C(n,f,\|u\|_{L^\infty(\Omega)})>0$.
Using \eqref{vai2}, the above inequality implies 
$$\int_{(\Omega_{\lambda^{**}}\cup R_{\lambda^{**}}\Omega_{\lambda^{**}})\setminus (K\cup R_{\lambda^{**}} K)}|\nabla w_{\lambda^{**}}|^2\,dx\leq 0.$$
Therefore, $w_\mu=0$ in $\Omega_\mu\setminus K$ and hence, $u\leq u_\mu$ in $\Omega_\mu\setminus R_\mu\Gamma$ for every $\mu\in(\lambda^{**},\lambda^*)$. This contradicts the definition of $\lambda^*$. Hence, $\lambda^*=0$. \\
\textbf{Step-III:} In this final step, we show $u(x_1,x_2,...,x_n)=u(-x_1,x_2,...,x_n)$ for every $(x_1,x_2,...,x_n)\in\Omega$ with $x_1>0$. In this regard, since $\lambda^*=0$, we have $u(x_1,x_2,...,x_n)\leq u(-x_1,x_2,...,x_n)$ for all $x=(x_1,x_2,...,x_n)\in \Omega$ with $x_1\geq0$. Moving the plane from left to right, a similar analysis leads us to conclude that $u(x_1,x_2,...,x_n)\geq u(-x_1,x_2,...,x_n)$ for all $x=(x_1,x_2,...,x_n)\in \Omega$ with $x_1\geq0$. Combining these two facts, we get our desired result. 

\textbf{Proof of Corollary \ref{Cor1}:} Since $\Omega$ is a ball in $\R^n$, it is convex and symmetric about every hyperplane that passes through the origin. Furthermore, $f$ satisfies all the assumptions of Theorem \ref{T1} with $\Gamma=\phi$ in every direction. Applying Theorem \ref{T1}, we deduce the result.

\textbf{Proof of Corollary \ref{Cor2}:} Since $0\leq u(x)\leq 1$ for every $x\in\Omega$, the R.H.S. of equation \eqref{ME2} satisfies the hypotheses of Theorem \ref{T1}. Consequently, the result follows from Theorem \ref{T1}.
\section{Proof of Theorem \ref{T2}}
 Before we proceed further, we mention a useful lemma. Recall that $\Sigma_\lambda$ is given in \eqref{M} and $\Omega_\lambda=\Omega\cap \Sigma_\lambda$.
\begin{Lemma}\label{PLD}
    Suppose $\Omega\subset\R^n$ is the unit ball. Let $0<\lambda\leq \lambda_1<1$ and $0<\eta<\frac{1-\lambda_1}{6}$. Assume that $v\in H^1(\Omega_\lambda)\cap \mathcal{L}^s(\R^n)$ is a non-negative weak supersolution to the following equation
    \begin{align*}
        -\De v+(-\De)^sv+cv=g \text{ in }\Omega_\lambda,
    \end{align*}
where $c\in L^\infty(\Omega)$ and $g\in L^n(\Omega)$. Then there exist constants $\mathcal{A}>0,\;\alpha>0,$ and $\epsilon>0$, depending only on $n,\;s,\;\|c\|_{L^\infty(\Omega)}$ and $\lambda_1$ such that
 \begin{align*}
     \sup_{y\in \Omega_{\lambda,\eta}}\Big(\fint_{B_\frac{\eta}{2}(y)}v^\epsilon\;dx\Big)^\frac{1}{\epsilon}\leq \mathcal{A}\eta^{-\alpha}\Big[\essinf_{\Omega_{\lambda,\eta}}v+\|g\|_{L^n(\Omega)}\Big],
 \end{align*}
 where
 \begin{align}\label{nbd}
     \Omega_{\lambda,\eta}:=\{x\in\Omega_\la: \mathrm{dist}(x,\partial\Omega_\lambda)>\eta\}.
 \end{align}
\end{Lemma}
\begin{proof}
The proof proceeds verbatim as in \cite[Corollary 2.3]{Cozzi24}. We omit it here.
\end{proof}
\textbf{Proof of Theorem \ref{T2}:} Suppose that $\Omega$ is the unit ball in $\R^n$ and $u\in H^1_0(\Omega)\cap C^1(\overline{\Omega})$ is a weak solution of the equation (\ref{ME3}) with $f(x,u)=k(x)g(u)$. Due to \cite[Theorem 2.7]{Valdinoci23}, there exists $\beta\in(0,1)$ such that $u\in C^{1,\beta}(\overline{\Omega})$. Furthermore, there exists $C>0$, depending only on $n,s,C_0,\|k\|_{L^\infty(\Omega)}$ and $\|g\|_{L^\infty([0,C_0])}$, such that 
\begin{align}\label{Bdd}
    \|Du\|_{L^\infty(\Omega)}+[Du]_{C^{0,\beta}(\Omega)}\leq C.
\end{align}
It is enough to prove that for every unit vector $\nu\in\R^n$ the following hold
$$|u(x)-u(x_\nu)|\leq \mathscr{C}\mathrm{def}(k)^\gamma \text{ for all }x\in \Omega \text{ with } x\cdot \nu>0,$$
for some constants $\mathscr{C}>0$ and $\gamma\in (0,1)$, depending only on $n,s,C_0,\|k\|_{L^\infty(\Omega)}$ and $ \|g\|_{C^{0,1}[0,C_0]}$. Here $x_\nu=x-2(x\cdot\nu)\nu$. We may assume $\mathrm{def}(k)$ is arbitrarily small; otherwise, the above claim is trivially satisfied. Without loss of generality, we assume that $\nu=(1,0,...,0)\in\R^n$ and aim to show
$$|u(x_1,x_2,...,x_n)-u(-x_1,x_2,...,x_n)|\leq \mathscr{C}\mathrm{def}(k)^\gamma, \text{ for all }x\in \Omega \text{ with } x_1>0.$$
In this regard, we define $$\Lambda:=\{ \lambda>0: \|(u-u_\mu)^+\|_{L^\infty(\Omega_\mu)}\leq \mathcal{C}\mathrm{def}(k),\text{ for all }\mu\in (\lambda,1)\},$$
where $\mathcal{C}=\max\{\mathcal{K},\Tilde{C}\}$. Here, the constants $\mathcal{K}$ and $\Tilde{C}$ are obtained in \textcolor{black}{\eqref{Sanjit}} and (\ref{Sanjit0}), respectively.\\
\textbf{Step-I:} Here, we show that the set $\Lambda$ is nonempty. To this end, we have 
\begin{align*}
    -\De u_\mu+(-\De)^su_\mu=k_\mu(x)g(u_\mu) \text{ in }\textcolor{black}{\Omega_\mu\cup R_\mu\Omega_\mu},
\end{align*}
where $k_\mu(x)=k(x_\mu)$. Thus, $w_\mu=u-u_\mu$ satisfies the following equation in the weak sense:
\begin{align}\label{MAAA}
     -\De w_\mu+(-\De )^sw_\mu+c_\mu(x)w_\mu=(k-k_\mu)g(u_\mu) \text{ in }\textcolor{black}{\Omega_\mu\cup R_\mu\Omega_\mu},
\end{align}
where 
\begin{align}\label{RS}
    c_\mu(x):=\begin{cases}
    -k(x)\frac{g(u)-g(u_\mu)}{u-u_\mu} &\text{ if } u\neq u_\mu,\\
    0&\text{otherwise.}  
\end{cases}
\end{align}
Since $k$ is bounded and $g$ is Lipschitz continuous in $[0,C_0]$, we have $c_\mu(x)\in L^\infty(\Omega)$. \textcolor{black}{It is immediate that the following function belongs to the space $H^1_0(\Omega_\mu\cup R_\mu\Omega_\mu)$: $$z_\mu=\begin{cases}
    (u-u_\mu)^+\mbox{ in }\Sigma_\mu,\\
    -(u-u_\mu)^-\mbox{ in }\R^n\setminus\Sigma_\mu.
\end{cases}$$
Note that $\|z_\mu\|_{L^\infty(\Omega_\mu\cup R_\mu\Omega_\mu)}=\|z_\mu\|_{L^\infty(\Omega_\mu)}=\|(u-u_\mu)^+\|_{L^\infty(\Omega_\mu)}$. We denote $\Tilde{z}_\mu:=z_\mu+d$ with $d:=\|(k-k_\mu)^+(x)g(u_\mu)\|_{L^\infty(\Omega_\mu)}$. Incorporating $\phi=\xi(\Tilde{z}_\mu)$ (see \ref{Babu1} for the definition of $\xi$) in the weak formulation of (\ref{MAAA}) and using the monotonicity of $\xi$, we deduce
\begin{align}
    \int_{\Omega_\mu}|\nabla \Tilde{z}_\mu|^2\xi'(\Tilde{z}_\mu)\;dx+\int_{R_\mu\Omega_\mu}|\nabla \Tilde{z}_\mu|^2\xi'(\Tilde{z}_\mu)\;dx \leq (1+\|c_\mu\|_{L^\infty(\Omega)})\int_{\Omega_\mu\cup R_\mu\Omega_\mu}\Tilde{z}_\mu\xi(\Tilde{z}_\mu)\;dx, 
\end{align}
which yields that
\begin{align}
   \int_{\Omega_\mu\cup R_\mu\Omega_\mu}|\nabla \Tilde{z}_\mu|^2\xi'(\Tilde{z}_\mu)\;dx \leq (1+\|c_\mu\|_{L^\infty(\Omega)})\int_{\Omega_\mu\cup R_\mu\Omega_\mu}\Tilde{z}_\mu\xi(\Tilde{z}_\mu)\;dx,  
\end{align}
Along the lines of proof of Lemma \ref{SML}, there exist constants $\delta>0$ and $\mathcal{K}>0$, depending on $n,s,\|k\|_{L^\infty(\Omega)},\|g\|_{C^{0,1}[0,C_0]}$ such that whenever 
\begin{align}\label{MAA1}
    |\Omega_\mu\cup R_\mu\Omega_\mu|\leq \delta
\end{align}
one has
$$\|z_\mu\|_{L^\infty(\Omega_\mu\cup R_\mu\Omega_\mu)}\leq \mathcal{K}\|(k-k_\mu)^+g(u_\mu)\|_{L^\infty(\Omega_\mu)},$$
which yields
\begin{align}
      \|(u-u_\mu)^+\|_{L^\infty(\Omega_\mu)}\leq \mathcal{K}\|(k-k_\mu)^+g(u_\mu)\|_{L^\infty(\Omega_\mu)}.
\end{align}}
Thus, there exist constants $\lambda_1\in (0,1)$ and $\mathcal{K}>0$, depending on $n,s,\|k\|_{L^\infty(\Omega)},\|g\|_{C^{0,1}[0,C_0]}$ such that for every $\lambda_1<\mu<1$ we deduce
\begin{align}\label{Sanjit}
      \|(u-u_\mu)^+\|_{L^\infty(\Omega_\mu)}\leq \mathcal{K}\|(k-k_\mu)^+g(u_\mu)\|_{L^\infty(\Omega_\mu)}.
\end{align}
Since $g\geq 0$, we get
\begin{align*}
    (k-k_\mu)^+(x)g(u_\mu)&\leq \|g\|_{C^{0,1}[0,C_0]} \big(k(x)-k(x_\mu^*)+k(x_\mu^*)-k(x_\mu)\big)\\
    &\leq \|g\|_{C^{0,1}[0,C_0]}\big(\|\nabla^Tk\|_{L^\infty(\Omega)}+\|\frac{\partial^+k}{\partial r}\|_{L^\infty(\Omega)}\big),
\end{align*}
where $x_\mu^*:=\frac{|x_\mu|}{|x|}x$. Therefore, (\ref{Sanjit}) yields that for every $\lambda_1<\mu<1$,
\begin{align}\label{Sanjit0}
    \|(u-u_\mu)^+\|_{L^\infty(\Omega_\mu)}\leq \Tilde{C}\mathrm{def}(k),
\end{align}
for some $\Tilde{C}=\Tilde{C}(n,\|k\|_{L^\infty(\Omega)},\|g\|_{C^{0,1}([0,C_0])})$. Consequently, $[\lambda_1,1)\subset\Lambda$, and hence $\Lambda\neq \emptyset$.\\
\textbf{Step-II:} In this step, we prove that $$\lambda_*:=\inf \Lambda\leq \frac{1}{4}.$$ We proceed by contraction. If possible, let $\frac{1}{4}<\lambda_*\leq \lambda_1<1$.
Due to the continuity of the solution and the definition of $\lambda_*$, we have $u-u_{\lambda_*}\leq \mathcal{C}\mathrm{def}(k)$ in $\Omega_{\lambda_*}$. Therefore, $v_{\lambda_*}:=u_{\lambda_*}-u+\mathcal{C}\mathrm{def}(k)$ is non-negative in $\Omega_{\lambda_*}$, and satisfies the following equation in the weak sense:
\begin{align*}
    -\De v+(-\De)^sv+c_{\la_*}^+v\geq-(k_{\lambda_*}-k)^+g(u_{\lambda_*})-\mathcal{C}c_{\lambda_*}^-(x)\mathrm{def}(k)\text{ in }\Omega_{\lambda_*},
\end{align*}
where $c_{\lambda_*}$ is defined in \eqref{RS} with $\mu=\lambda_*$. Since $p_0=(1-2\eta,0,0,...,0)\in \Omega_{\lambda_*,\eta}$, due to Lemma \ref{PLD}, for $0<\eta<\frac{1-\lambda_1}{6}$ we have
\begin{align}\label{Con1}
    \Big(\fint_{B_{\frac{\eta}{2}}(p_0)}v_{\lambda_*}^\epsilon\Big)^\frac{1}{\epsilon}&\leq \mathcal{A}\eta^{-\alpha}\Big(\inf_{\Omega_{\lambda_*,\eta}}v_{\lambda_*}+\|(k_{\lambda_*}-k)^+g(u_{\lambda_*})+\mathcal{C}c^-(x)\mathrm{def}(k)\|_{L^n(\Omega)}\Big)\nonumber\\
    &\leq C\eta^{-\alpha}\Big(\inf_{\Omega_{\lambda_*,\eta}}(u_{\lambda_*}-u)+\mathrm{def}(k)\Big),
\end{align}
where $C=C(n,s,\|g\|_{C^{0,1}([0,C_0])},\|k\|_{L^\infty(\Omega)})>0$ is a constant and $\epsilon>0$ is obtained in Lemma \ref{PLD}. Without loss of generality, we may assume the constants in (\ref{Bdd}) and (\ref{Con1}) are the same, otherwise one can replace both of them by their maximum value. Now, using \eqref{Bdd} and (\ref{lowerbound}), we get
\begin{align}\label{Sreeja1}
    \Big(\fint_{B_{\frac{\eta}{2}}(p_0)}v_{\lambda_*}^\epsilon\Big)^\frac{1}{\epsilon}\geq \inf_{B_{\frac{\eta}{2}(p_0)}}v_{\lambda_*}&\geq \inf_{B_{\frac{\eta}{2}(p_0)}} (u_{\lambda_*}-u)\nonumber\\
    &=\inf_{x\in B_{\frac{\eta}{2}(p_0)}} (u(x_{\lambda_*})-(u(x)-u(\Tilde{x})))\nonumber\\
    &\geq \frac{1}{C_0}(1-|x_{\lambda_*}|)-\|\nabla u\|_{L^\infty(\Omega)}|x-\Tilde{x}|\nonumber\\
    &\geq \frac{1}{C_0}(1-|x_{\lambda_*}|)-\frac{5C}{2}\eta\nonumber\\
    &\geq \frac{1}{C_0} \min\{{ 1-|2\lambda_*-1+\frac{5\eta}{2}|, 1-|2\lambda_*-1+\frac{3\eta}{2}|\}}-\frac{5C}{2}\eta\nonumber\\
    &\geq \frac{(1-\lambda_1)}{2C_0}-\frac{5C}{2}\eta,
\end{align}
where $\Tilde{x}$ denotes the intersecting point of the line passing through the point $x$ along the direction $(1,0,...,0)$ and $\partial \Omega$. Combining \eqref{Con1} and \eqref{Sreeja1}, we obtain
\begin{align*}
    \inf_{\Omega_{\lambda_*,\eta}}(u_{\lambda_*}-u)\geq \Big( \frac{(1-\lambda_1)}{2C_0}-\frac{5C}{2}\eta\Big)\frac{\eta^\alpha}{C}-\mathrm{def}(k).
\end{align*}
Choose $\eta<\eta_1:=\frac{(1-\lambda_1)}{10C_0C}$, $\mathrm{def}(k)\leq \gamma_1:=\frac{\eta^\alpha(1-\lambda_1)}{8C_0C}$ and hence
\begin{align*}
     \inf_{\Omega_{\lambda_*,\eta}}(u_{\lambda_*}-u)\geq \frac{\eta^\alpha(1-\lambda_1)}{8C_0C}.
\end{align*}
For $0<\sigma< \frac{\eta^\alpha(1-\lambda_1)}{32C^2C_0}$, we have
\begin{align}\label{Sreeja2}
   (u_{\lambda_*-\sigma}-u)=(u_{\lambda_*-\sigma}-u_{\lambda_*}+u_{\lambda_*}-u)&\geq  \frac{\eta^\alpha(1-\lambda_1)}{8CC_0}-2\sigma\|\nabla u\|_{L^\infty(\Omega)}\nonumber\\
   &\geq  \frac{\eta^\alpha(1-\lambda_1)}{8CC_0}-2\sigma C\nonumber\\
   &\geq  \frac{\eta^\alpha(1-\lambda_1)}{16CC_0}\geq 0 \text{ in }\Omega_{\lambda^*,\eta}.
\end{align}
\textcolor{black}{Now, we denote $U_{\sigma,\eta}:=({\Omega}_{\lambda_*-\sigma}\cup R_{\lambda_*-\sigma}{\Omega}_{\lambda_*-\sigma})\setminus({\Omega}_{\lambda_*,\eta}\cup R_{\lambda_*}{\Omega}_{\lambda_*,\eta})$} and it is immediate to show the existence of a positive constant $\tau=\tau(n)>0$ such that 
\begin{align}\label{small}
    |U_{\sigma,\eta}|\leq \tau\eta,
\end{align}
for every $0<\sigma\leq \eta$. It is clear that $v_{\lambda_*-\sigma}:=(u-u_{\lambda_*-\sigma})$ is a weak solution of the equation 
\begin{align}\label{MA}
\textcolor{black}{-\De v_{\lambda_*-\sigma}+(-\De)^sv_{\lambda_*-\sigma}+c_{\lambda_*-\sigma}v_{\lambda_*-\sigma}=(k_{\lambda_*-\sigma}-k)g(u_{\lambda_*-\sigma}) \text{ in } U_{\sigma,\eta}}.
\end{align}
\textcolor{black}{Proceeding as in (\ref{Sanjit}), we can choose $\sigma,\eta>0$ small enough such that $|U_{\sigma,\eta}|\leq \delta$, where $\delta>0$ is obtained in \eqref{MAA1}. Consequently,
\begin{align*}
    \|v_{\lambda_*-\sigma}^+\|_{L^\infty(U_{\sigma,\eta})}\leq \mathcal{K}\mathrm{def}(k).
\end{align*}}
This fact, together with \eqref{Sreeja2} yields 
\begin{align}\label{contra}
    \|(u-u_{\lambda_*-\sigma})^+\|_{L^\infty(\Omega_{\lambda_*-\sigma})}\leq \mathcal{C} \mathrm{def}(k), \text{ for small enough } \sigma>0,
\end{align}
which violates the definition of $\lambda_*$. Consequently, $\lambda_*\leq \frac{1}{4}.$\\
\textbf{Step-III:} Here, we show that
\begin{align*}
    \lambda_*\leq (2CC_0C_1^\alpha\mathrm{def}(k))^\frac{1}{1+\alpha},
\end{align*}
where $C_1:=\max\{\frac{5CC_0}{2}, \frac{\tau}{4\delta}\}$ \textcolor{black}{and the constant $C$ is defined in (\ref{Con1}).} If possible, let $(2CC_0C_1^\alpha\mathrm{def}(k))^\frac{1}{1+\alpha}<\lambda_*\leq \frac{1}{4}$. Since $\lambda_*\leq \frac{1}{4}$, proceeding as in (\ref{Sreeja1}), we obtain that for every $0<\eta<\frac{1}{8}$, 
\begin{align}
     \Big(\fint_{B_{\frac{\eta}{2}(p_0)}}v_{\lambda_*}^\epsilon\Big)^\frac{1}{\epsilon}&\geq \frac{1}{C_0} \min\{{ 1-|2\lambda_*-1+\frac{5\eta}{2}|, 1-|2\lambda_*-1+\frac{3\eta}{2}|\}}-\frac{5C}{2}\eta\nonumber\\
     &\geq \frac{2\lambda_*}{C_0}-\frac{5C}{2}\eta,
\end{align}
where in the second inequality, we used the fact $\lambda_*\leq \frac{1}{4}$. Using \eqref{Con1}, for $0<\eta<\frac{1}{8}$ one has 
\begin{align*}
    \inf_{\Omega_{\lambda_*,\eta}}(u_{\lambda_*}-u)\geq \Big( \frac{2\lambda_*}{C_0}-\frac{5C}{2}\eta\Big)\frac{\eta^\alpha}{C}-\mathrm{def}(k).
\end{align*}
Choosing $\eta=\frac{\lambda_*}{C_1}$ to deduce
\begin{align}\label{Sreeja3}
     \inf_{\Omega_{\lambda_*,\eta}}(u_{\lambda_*}-u)\geq \frac{\lambda_*^{1+\alpha}}{CC_0C_1^\alpha}-\mathrm{def}(k).
\end{align}
Since $(2CC_0C_1^\alpha\mathrm{def}(k))^\frac{1}{1+\alpha}<\lambda_*$, (\ref{Sreeja3}) yields
\begin{align}\label{Sreeja4}
     \inf_{\Omega_{\lambda_*,\eta}}(u_{\lambda_*}-u)\geq \frac{\lambda_*^{1+\alpha}}{2CC_0C_1^\alpha}\geq 0.
\end{align}
Consequently, proceeding as in (\ref{Sreeja2}), one can choose small enough $\sigma>0$ such that
\begin{align*}
    u_{\lambda_*-\sigma}-u\geq \frac{\lambda_*^{1+\alpha}}{2CC_0C_1^\alpha}-2\sigma C\geq 0\text{ in }\Omega_{\lambda_*,\eta},
\end{align*}
Since $\eta\leq \frac{\delta}{\tau}$, (\ref{small}) yields $|U_{\sigma,\eta}|\leq \delta$, where $\delta$ is given in \textcolor{black}{\eqref{MAA1}}. Hence, by an argument analogous to that used to derive \textcolor{black}{\eqref{contra}}, we arrive at a contradiction. Consequently, we conclude that 
\begin{align*}
    \lambda_*\leq (2CC_0C_1^\alpha\mathrm{def}(k))^\frac{1}{1+\alpha}.
\end{align*}
\textbf{Step-IV:} In this step, we show that there exist positive constants $C$ and $\gamma$, depending on $n,s,\|k\|_{L^\infty(\Omega)},$ $\|g\|_{C^{0,1}[0,C_0]}$ such that $$|u(x_1,x')-u(-x_1,x')|\leq C\mathrm{def}(k)^\gamma,\text{ for all }x=(x_1,x')\in B_1\text{ with }x_1>0.$$ To this end, using the previous step, one has
$$\|(u-u_\lambda)^+\|_{L^\infty(\Omega_\lambda)}\leq \mathcal{C}\mathrm{def}(k),\text{ for every }\lambda\in[\lambda_2,1),$$
where $\lambda_2:=(2CC_0C_1^\alpha\mathrm{def}(k))^\frac{1}{1+\alpha}$ \textcolor{black}{with $C$ is given in inequality (\ref{Con1})}. If $x\in \Omega_{\lambda_2}$, we get
\begin{align*}
    u(x_1,x')-u(-x_1,x')&=u(x_1,x')-u_{\lambda_2}(x_1,x')+u_{\lambda_2}(x_1,x')-u(-x_1,x')\nonumber\\
    &\leq \|(u-u_{\lambda_2})^+\|_{L^\infty(\Omega_{\lambda_2})}+2\lambda_2\|\nabla u\|_{L^\infty(\Omega)}\nonumber\\
    &\leq \mathcal{C}\mathrm{def}(k)+C(2CC_0C_1^\alpha\mathrm{def}(k))^\frac{1}{1+\alpha}\leq C\mathrm{def}(k)^\frac{1}{1+\alpha},
\end{align*}
for some constant $C=C(n,s,C_0,\|k\|_{L^\infty(\Omega)},\|g\|_{C^{0,1}([0,C_0])})>0$.
For $x\in \Omega_0\setminus\Omega_{\lambda_2}$, we have
\begin{align*}
    u(x_1,x')-u(-x_1,x')&=u(x_1,x')-u(0,x')+u(0,x')-u(-x_1,x')\nonumber\\
    &\leq 2x_1\|\nabla u\|_{L^\infty(\Omega)}\leq 2\lambda_2\|\nabla u\|_{L^\infty(\Omega)}\leq C\mathrm{def}(k)^\frac{1}{1+\alpha},
\end{align*}
for some constant $C=C(n,s,C_0,\|k\|_{L^\infty(\Omega)},\|g\|_{C^{0,1}([0,C_0])})>0$. Consequently, we obtain
\begin{align}\label{Ma}
    u(x_1,x')-u(-x_1,x')\leq C\mathrm{def}(k)^\frac{1}{1+\alpha} \mbox{ for every } x=(x_1,x')\in \Omega \mbox{ with }x_1>0,
\end{align}
for some constant $C=C(n,s,C_0,\|k\|_{L^\infty(\Omega)},\|g\|_{C^{0,1}([0,C_0])})>0$. By a similar argument, we can deduce
\begin{align}\label{Maa}
    u(-x_1,x')-u(x_1,x')\leq C\mathrm{def}(k)^\frac{1}{1+\alpha}, \mbox{ for every } x=(x_1,x')\in \Omega \mbox{ with }x_1>0,
\end{align}
for some constant $C>0$, depending on $n,s,C_0,\|k\|_{L^\infty(\Omega)}, \|g\|_{C^{0,1}[0,C_0]}$. Combining (\ref{Ma}) and (\ref{Maa}), we obtain our desired estimate with $\gamma=\frac{1}{1+\alpha}$.
\section*{Acknowledgment}
This work was carried out during the author's FARE (Fellowship for Academic and Research Excellence) tenure at the Indian Institute of Technology Kanpur. The author gratefully acknowledges the financial support from the FARE Fellowship (ID: FA2408013) and thanks the Institute for awarding the fellowship.  Furthermore, he thanks the anonymous reviewer for his suggestions and comments, which have significantly improved the clarity of the manuscript.
\nocite{*}
\bibliographystyle{plain}
\bibliography{Reference.bib}
Sanjit Biswas\\
Department of Mathematical Sciences\\
Indian Institute of Science Education and Research Berhampur\\
Berhampur, Odisha-760003, India\\
Email: sanjitbiswas410@gmail.com
\end{document}